\documentclass[final,onefignum,onetabnum,dvipsnames,svgnames]{siamart190516}

\ifpdf
  \DeclareGraphicsExtensions{.eps,.pdf,.png,.jpg}
\else
  \DeclareGraphicsExtensions{.eps}
\fi

\newsiamremark{remark}{Remark}
\newsiamremark{hypothesis}{Hypothesis}
\crefname{hypothesis}{Hypothesis}{Hypotheses}
\newsiamthm{claim}{Claim}

\headers{Composite nonsmooth optimization}{G. Bareilles, F. Iutzeler, and J. Malick}

\title{Harnessing structure\\ in composite nonsmooth minimization\thanks{Submitted to the editors June 27, 2022.}}

\author{Gilles Bareilles\thanks{Univ.\;Grenoble\;Alpes, LJK, France
  (\email{gilles.bareilles@univ-grenoble-alpes.fr}, \email{franck.iutzeler@univ-grenoble-alpes.fr}).}
\and Franck Iutzeler\footnotemark[2]
\and J\'er\^ome Malick\thanks{Univ.\;Grenoble\;Alpes, CNRS, Grenoble\;INP, LJK\;(\email{jerome.malick@univ-grenoble-alpes.fr}).}}

\usepackage{lipsum}
\usepackage{afterpage}

\usepackage{amsmath}						%
\usepackage{amssymb}						%
\usepackage{amsfonts}						%
\usepackage{array}

\usepackage{mathtools}						%

\usepackage[utf8]{inputenc}							%
\usepackage[T1]{fontenc}						%
\usepackage[normalem]{ulem}

\colorlet{MyBlue}{DodgerBlue!75!Black}
\colorlet{MyGreen}{DarkGreen!85!Black}
\colorlet{MyGray}{White!75!Black}

\definecolor{chartreuse}{HTML}{288800}
\definecolor{dark_gray}{HTML}{111111}
\definecolor{light_gray}{HTML}{555555}
\definecolor{tomato}{HTML}{FF6347}
\definecolor{snow}{HTML}{F1F1F1}
\definecolor{whiteee}{HTML}{EEEEEC}

\definecolor{c1}{RGB}{238,102,119}
\definecolor{c2}{RGB}{68, 119, 170}
\definecolor{c3}{RGB}{102, 204, 238}
\definecolor{c4}{RGB}{34, 136, 51}
\definecolor{c5}{RGB}{204, 187, 68}
\definecolor{c6}{RGB}{238, 102, 119}
\definecolor{c7}{RGB}{170, 51, 119}
\definecolor{c8}{RGB}{187, 187, 187}

\usepackage{subcaption}

\usepackage{tikz}							%
\usetikzlibrary{calc,patterns,arrows}
\usetikzlibrary{tikzmark,shapes}
\usetikzlibrary{decorations.pathreplacing}
\usetikzlibrary{matrix}
\usetikzlibrary{perspective}
\usetikzlibrary{external}
\usepackage{pgfplots}
\usepgfplotslibrary{groupplots}

\usepackage{multicol}
\usepackage{arydshln}
\usepackage{multirow}
\usepackage{wrapfig}

\usepackage{algorithm}
\usepackage[noend]{algpseudocode}

\def\addlegendimage{\csname pgfplots@addlegendimage\endcsname}

\usepackage{acronym}						%
\usepackage{booktabs}						%
\usepackage{latexsym}						%
\usepackage{paralist}						%
\usepackage{xspace}							%
\usepackage{placeins}
\usepackage{enumitem}

\usepackage{hyperref}
\hypersetup{
final,
colorlinks=true,
linktocpage=true,
pdfstartview=FitH,
breaklinks=true,
pdfpagemode=UseNone,
pageanchor=true,
pdfpagemode=UseOutlines,
plainpages=false,
bookmarksnumbered,
bookmarksopen=false,
bookmarksopenlevel=1,
hypertexnames=true,
pdfhighlight=/O,
urlcolor=Maroon,linkcolor=MyBlue!60!black,citecolor=DarkGreen!70!black,	%
pdftitle={},
pdfauthor={},
pdfsubject={},
pdfkeywords={},
pdfcreator={pdfLaTeX},
pdfproducer={LaTeX with hyperref}
}

\usepackage{autonum}

\newtheorem{assumption}[theorem]{Assumption}					%
\crefname{assumption}{assumption}{assumptions}
\newtheorem{property}[theorem]{Property}					%
\crefname{property}{property}{properties}

\newtheorem{example}[theorem]{Example}						%

\usepackage{xstring}
\usepackage{catchfile}

\newcommand{\debug}[1]{#1}					%

\usepackage[suppress]{color-edits}
\addauthor{GB}{black}
\addauthor{FI}{black}
\addauthor{JM}{black}
\addauthor{GBtwo}{blue}

\newcommand{\ie}{\emph{i.e.}\xspace}
\newcommand{\eg}{\emph{e.g.}\xspace}

\newcommand{\newmacro}[2]{\newcommand{#1}{\debug{#2}}}		%
\newcommand{\newop}[2]{\DeclareMathOperator{#1}{\debug{#2}}}		%

\newop{\argmin}{argmin}
\newop{\argmax}{argmax}
\newop{\Diag}{Diag}
\newop{\diag}{diag}
\newop{\supp}{supp}
\newop{\nullC}{null}
\newop{\trace}{trace}
\newop{\dist}{dist}
\newop{\rank}{rank}
\newop{\Aff}{Aff}
\newop{\Conv}{Conv}
\newop{\sign}{sign}
\newop{\proj}{proj}
\newop{\ri}{ri}
\newop{\intr}{int}
\newop{\cl}{cl}
\newop{\dom}{dom}
\newop{\Id}{I}
\renewcommand{\Im}{\mathrm{Im}}

\newmacro{\ball}{ \mathcal{B}}

\newmacro{\dd}{\mathrm{d}} %

\DeclarePairedDelimiterX{\product}[2]{\debug{\langle}}{\debug{\rangle}}{#1,#2}		%

\newmacro{\defeq}{\triangleq}

\usepackage{xifthen}%

\newmacro{\RR}{\mathbb{R}}
\newmacro{\bbR}{\mathbb{R}}
\newmacro{\ndim}{n}
\newmacro{\Rn}{\RR^{\ndim}}
\newmacro{\RRext}{ \overline{\mathbb{R}}} %

\newmacro{\rdim}{\ndim_1}
\newmacro{\cdim}{\ndim_2}
\newmacro{\Rmat}{\RR^{\rdim\times\cdim}}

\newmacro{\NN}{\mathbb{N}}

\newmacro{\vx}{x}
\newmacro{\vy}{y}
\newmacro{\vxalt}{u}
\newmacro{\vyalt}{z}
\newmacro{\vxman}{\vx^{\M}}

\newmacro{\matx}{X}
\newmacro{\maty}{Y}

\newmacro{\vxbase}{ \bar{\vx} }
\newmacro{\vybase}{ \bar{\vy} }

\newmacro{\vsg}{  v}
\newmacro{\vsgbase}{ \bar{\vsg}}

\newmacro{\fixedp}{\bar{\vx}}
\newmacro{\fixedpalt}{\bar{\vxalt}}

\newmacro{\critp}{\bar{\vx}}

\newmacro{\B}{  \mathcal{B}}
\newmacro{\N}{  \mathcal{N}}

\newmacro{\nhd}{\mathcal{U}}

\newmacro{\Cfun}{G}
\newmacro{\impfun}{\mathrm{p}}

\ifdef{\C}{                          %
    \renewcommand{\C}{ \debug{\mathcal{C}}}
}{
    \newcommand{\C}{  \debug{\mathcal{C}}}
}

\newmacro{\fun}{  F}   %
\newmacro{\funalt}{ \tilde{F}}   %
\newmacro{\funman}{  f} %
\newmacro{\funs}{  f}   %
\newmacro{\funns}{  g}  %
\newmacro{\funcomp}{  F}  %
\newmacro{\mapping}{ c}

\newop{\codim}{codim}
\newmacro{\Mcodim}{k}  %

\newmacro{\maneq}{h}

\newmacro{\set}{ C}  %

\newmacro{\cvxparam}{ \alpha}  %

\newcommand{\opt}[1][\vx]{\debug{{#1}^\star}}		%

\newmacro{\step}{\gamma}
\newmacro{\iner}{\alpha}

\newmacro{\prox}{\mathbf{prox}}

\newmacro{\proxstep}{\step}
\newmacro{\PGop}{\mathsf{T}} %
\newmacro{\surr}{\rho}
\newmacro{\moreau}{e}

\newmacro{\ind}{ \mathrm{ind}}
\newmacro{\indicator}{ \iota}

\newmacro{\Lip}{L}
\newmacro{\Scvx}{\mu}

\newmacro{\rpb}{r_{pb}}
\newmacro{\rpr}{r}

\ifdef{\M}{                          %
    \renewcommand{\M}{ \debug{\mathcal{M}}}
}{
    \newcommand{\M}{  \debug{\mathcal{M}}}
}

\newmacro{\Mcrit}{\debug{\M_{\critp}}}
\newmacro{\dimman}{p}

\newcommand{\distM}[1][\M]{\dist_{#1}}
\newcommand{\projM}[1][\M]{\proj_{#1}}

\newmacro{\locparam}{\varphi}
\newmacro{\loceq}{\Phi}

\newmacro{\smoothcurve}{  c}
\newmacro{\geocurve}{  \gamma}

\newmacro{\grad}{  \operatorname{grad}}
\newmacro{\Hess}{  \operatorname{Hess}}

\newmacro{\R}{  \operatorname{R}}           %
\newmacro{\D}{  \operatorname{D}}           %
\newmacro{\Jac}{  \operatorname{Jac}}           %

\newmacro{\tang}{\eta} %

\newmacro{\tangentBundle}{  T \mathcal B}

\newcommand{\projT}[2]{
  \proj_{\tangent{#1}{#2}}
 }

\newcommand{\tangent}[2]{  \debug{T}_{#1} #2}
\newcommand{\tangentM}[1][]{%
  \ifthenelse{\isempty{#1}}{  \debug{T}_{\vx} \M}{  \debug{T}_{#1} \M}%
}

\newcommand{\tangentMstar}[1][]{%
  \ifthenelse{\isempty{#1}}{  \debug{T}_{\vx} \M^\star}{  \debug{T}_{#1} \M^\star}%
}

\newcommand{\normal}[2]{  \debug{N}_{#1} #2}
\newcommand{\normalM}[1][]{%
  \ifthenelse{\isempty{#1}}{  \debug{N}_{\vx} \M}{  \debug{N}_{#1} \M}%
}

\newmacro{\nM}{q}
\newmacro{\col}{\mathsf{C}}
\newmacro{\collong}{\mathsf{C}=\{\M_1,\ldots,\M_\nM\}}

\newmacro{\Str}{\mathsf{S}}

\newmacro{\ite}{k}
\newmacro{\initite}{0}
\newmacro{\afterinitite}{1}
\newmacro{\sumite}{n}
\newmacro{\Afterite}{K}

\newcommand{\prev}[1][\vx]{\debug{#1_{\ite\debug{-1}}}}		%
\newcommand{\curr}[1][\vx]{\debug{#1_{\ite}}}				%
\renewcommand{\next}[1][\vx]{\debug{#1_{\ite\debug{+1}}}}		%

\newmacro{\itealt}{\ell}

\newcommand{\init}[1][\vx]{\debug{#1_{\initite}}}			%

\newcommand{\Fsext}[1][\funcomp]{\debug{\tilde{#1}}} %

\newcommand{\currinter}[1][\vx]{\debug{#1^{\funns}_{\ite}}}
\newcommand{\currFsext}[1][\vx]{\debug{\tilde{#1}_{\ite}}}

\newmacro{\bigoh}{\mathcal{O}}
\newmacro{\linrate}{\beta}

\newmacro{\Prob}{\mathbb{P}}

\newmacro{\EE}{\mathbb{E}}
\newmacro{\FF}{\mathcal{F}}

\newmacro{\indata}{\mathbf{a}}
\newmacro{\outdata}{b}
\newmacro{\inset}{\mathcal{A}}
\newmacro{\outset}{\mathcal{B}}

\newmacro{\Distrib}{\mathcal{D}}

\newmacro{\pred}{p}
\newmacro{\predfun}{P}

\newmacro{\loss}{\ell}

\newmacro{\nSamples}{m}
\newmacro{\dataset}{\mathcal{S}_\nSamples}

\newmacro{\Fmapping}{\Phi}

\newmacro{\paramspace}{\Theta}

\newmacro{\risk}{R}
\newmacro{\regparam}{\lambda}
\newmacro{\reg}{\Omega}

\newmacro{\past}{g}

\newmacro{\activ}{\sigma}

\newmacro{\inputSpace}{\Rn}

\newmacro{\ndimalt}{m}
\newmacro{\Rnalt}{\RR^{\ndimalt}}
\newmacro{\interSpace}{\Rnalt}

\newmacro{\vxinter}{\vy}
\newmacro{\Minter}{\M^{\funns}}
\newmacro{\Mopt}{\opt[\M]}
\newmacro{\Mr}{\M^{\lambda_{\max}}_{\mult}}
\newmacro{\MI}{\M^{\max}_I}

\newmacro{\lammax}{\lambda_{\max{}}}
\newmacro{\mult}{r}

\newcommand{\crit}[1][\vx]{\debug{\bar{#1}}}				%

\newcommand{\Sym}[1][m]{\debug{\mathbb{S}_{#1}}}				%

\newop{\bnd}{bnd}
\newop{\rbd}{rbd}

\usepackage{accents}                                  %
\newmacro{\stepLow}{\underaccent{\bar}{\step}}
\newmacro{\stepUp}{\bar{\step}}

\newmacro{\Lagmult}{\lambda}

\newcommand{\dSQP}[1][\vx]{d^{\mathrm{SQP}}(#1)}
\newcommand{\dcorr}[1][\vx]{d^{\mathrm{corr}}(#1)}
\newcommand{\currdSQP}{d_{\ite}^{\mathrm{SQP}}(\curr[\vx])}

\makeatletter
\providecommand{\leftsquigarrow}{%
  \mathrel{\mathpalette\reflect@squig\relax}%
}
\newcommand{\reflect@squig}[2]{%
  \reflectbox{$\m@th#1\rightsquigarrow$}%
}
\makeatother

\newmacro{\const}{C}
\newmacro{\constalt}{C'}
\newmacro{\constcurve}{\tilde{L}}
\newmacro{\conststep}{C''}
\newmacro{\Tcurve}{T}

\newmacro{\cm}{ e}
\newcommand{\cri}{\debug{c_{\text{ri}}}}
\newcommand{\cmap}{\debug{c_{\text{map}}}}

\newmacro{\vn}{v_n}
\newmacro{\vnalt}{\hat{v}_n}

\newmacro{\stepLowBnd}{\varphi}
\newmacro{\stepUpBnd}{\Gamma}

\newcommand{\includetikzfig}[2][1]{
  \resizebox{#1\textwidth}{!}{
    \includegraphics{#2.pdf}
  }
}
\newcommand{\includetikzfigcaption}[3][1]{
  \resizebox{#1\textwidth}{!}{
    \includegraphics{#2.pdf}
  }
  \caption{#3}
}
\newcommand{\includetikzexpenum}[2][1]{
  \resizebox{#1\textwidth}{!}{
    \includegraphics{#2.pdf}
  }
}
\newcommand{\includetikzexpenumcaption}[3][1]{
  \resizebox{#1\textwidth}{!}{
    \includegraphics{#2.pdf}
  }
  \caption{#3}
}

\ifpdf
\hypersetup{
  pdftitle={Harnessing structure in composite nonsmooth minimization},
  pdfauthor={G. Bareilles, F. Iutzeler, and J. Malick}
}
\fi

\begin{document}

\maketitle

\begin{abstract}
  We consider the problem of minimizing the composition of a nonsmooth function with a smooth mapping in the case where the proximity operator of the nonsmooth function can be explicitly computed.
  We first show that this proximity operator can provide the exact smooth substructure of minimizers, not only of the nonsmooth function, but also of the full composite function.
  We then exploit this proximal identification by proposing an algorithm which combines proximal steps with sequential quadratic programming steps.
  We show that our method \GBedit{locally} identifies the optimal smooth substructure and \FIedit{then} converges \GBreplace{locally }{}quadratically.
  We illustrate its behavior on two problems: the minimization of a maximum of quadratic functions and the minimization of the maximal eigenvalue of a parametrized matrix.
\end{abstract}



\section{Introduction}

\subsection{Context: structured nonsmooth optimization}

In this paper, we consider nonsmooth optimization problems of the form
\begin{align}\label{eq:minF}
  \min_{\vx\in\Rn} ~\funcomp(\vx) \;\defeq \;\funns ( \mapping (\vx)),
\end{align}
where the inner mapping $\mapping: \Rn \to \Rnalt$ is smooth and the outer function $\funns: \Rnalt \to \RR\cup\{+\infty\}$ is nonsmooth and may be nonconvex, but admits an explicit proximity operator.
Such composite nonsmooth optimization problems appear in a variety of applications in signal processing, machine learning, and control, such as robust nonlinear regression, phase synchronization, nonsmooth penalty functions; see \eg\cite{lewis2016proximal,shapiroClassNonsmoothComposite2003} and the references therein.

Throughout the paper, we illustrate our developments on two classes of functions: the pointwise maximum of $m$ smooth real-valued functions $\mapping_i$
\begin{align}\label{example:maxofsmooth}
  \funcomp(\vx) = \!\!\max_{i=1,\ldots, m} (\mapping_i(\vx))
\end{align}
and the maximum eigenvalue of a parametrized symmetric real matrix $\mapping$ 
\begin{align}\label{example:eigmax}
  \funcomp(\vx) = \lammax (\mapping(\vx)).
\end{align}
In these two examples and many others, subgradients of $\funcomp$ can be computed and thus the composite function can be minimized using standard nonsmooth optimization algorithms (\eg\!subgradient methods, gradient sampling\;\cite{burke2020gradient}, nonsmooth BFGS\;\cite{lewisNonsmoothOptimizationQuasiNewton2013}, or bundle methods\;\cite{hiriart-lemarechal-1993}). 
Nevertheless, these methods do not exploit the fact that $\funcomp$ is a composition of a smooth mapping $\mapping$, 
which can hinder their performance. 
In contrast, the so-called prox-linear methods leverage this composite expression by introducing an extension of the proximity operator where the nonlinear mapping $\mapping$ is iteratively replaced by a first-order Taylor approximation~\cite{lewis2016proximal}.
These methods benefit from theoretical convergence guarantees, and nicely generalize to Taylor-like approximations~\cite{drusvyatskiyNonsmoothOptimizationUsing2021,bolteMultiproximalLinearizationMethod2020}.
However these methods are not always directly implementable because the prox-linear step may be hard to compute, as in~\eqref{example:eigmax}. %

In this paper, we propose an optimization algorithm for solving\;\eqref{eq:minF} exploiting that the nonsmooth objective function $\funcomp=\funns\circ\mapping$ writes as a composition \GBreplace{with a}{between a smooth mapping $\mapping$ and a} simple nonsmooth function $\funns$ which displays some smooth substructure, as discussed below.

\subsection{Smooth substructure, identification, and existing algorithms}

For many composite functions, including \eqref{example:maxofsmooth} and \eqref{example:eigmax}, the nondifferentiability points \FIedit{\emph{locally}} organize into \emph{smooth manifolds over which %
$\funcomp$ evolves smoothly}.
We illustrate in \Cref{fig:pairspaces} such a smooth substructure for a maximum of two functions.

\begin{figure}[h!]
  \centering
\medskip 
  \begin{subfigure}[t]{0.49\textwidth}
    \centering
    \includetikzfigcaption{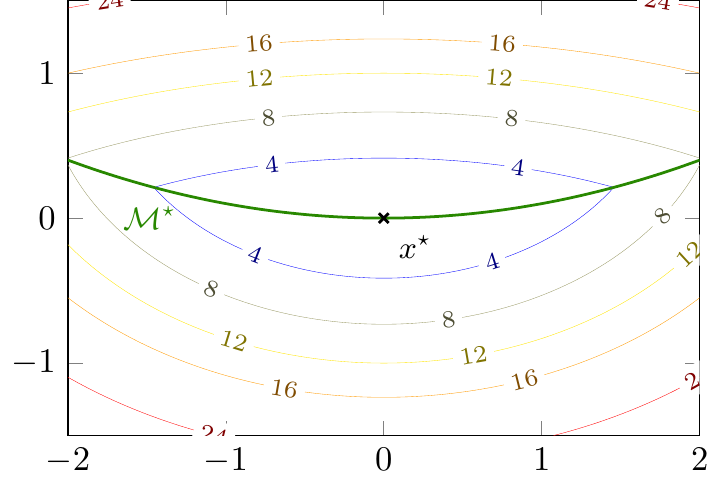}{\vspace*{-2ex}Level curves of $\funcomp$\label{fig:pairInputSpace}}
  \end{subfigure}
  \hfill
  \begin{subfigure}[t]{0.49\textwidth}
    \centering
    \includetikzfigcaption{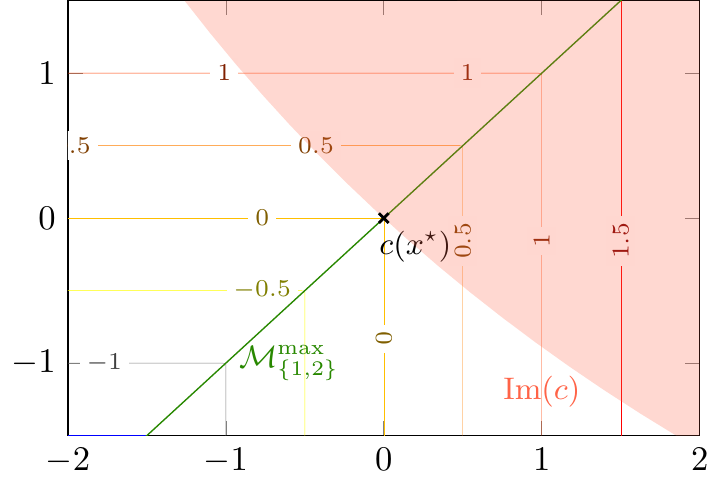}{\vspace*{-2ex}Level curves of $\funns$ and image of $\mapping$\label{fig:pairInterSpace}}
  \end{subfigure}
  \vspace*{-2ex} 
  \caption{
    Smooth substructure on a simple example $(n = m = 2)$ \GBedit{with $\funns(y) = \max(y_{1}, y_{2})$ and $c(x) = (2.6\, x_{1}^{2} + 4\,(x_{2}-1)^{2} - 4, x_{1}^{2} + 4\,(x_{2}+1)^{2} - 4)$.
    The right-hand figure shows the level curves of $\funns$ %
    (in the intermediate space), and the left-hand figure shows the ones of $\funcomp$
    (in the input space).} The manifolds of non-differentiability \GBedit{of $F$ and $g$} are in green. \GBedit{The right figure also displays} the image of $c$ \GBreplace{is}{as} the red area.%
    \label{fig:pairspaces}
  }
  \vspace*{-2ex} 
\end{figure}

The smooth substructure of $\funcomp$ can help in solving \eqref{eq:minF}. %
Indeed, if the optimal solution\;$\opt[x]$ belongs to a manifold $\opt[\M]$ that is known beforehand, then minimizing the nonsmooth function $\funcomp$ over $\inputSpace$ boils down to minimizing the smooth restriction $\funcomp|_{\opt[\M]}$ over this smooth \emph{optimal manifold} $\opt[\M]$.
This would enable to solve \eqref{eq:minF} by smooth constrained optimization algorithms, such as Sequential Quadratic Programming (SQP) methods (see \eg\cite{nocedal2006numerical,bonnans2006numerical}).\footnote{
  \FIedit{
    Note that $\opt[\M]$ is an arbitrary manifold and thus computing a feasible point is already a difficult task in general.
    That is why we consider in this paper infeasible methods, such as SQP, instead of feasible ones, like the Riemannian Newton algorithm.
  }
}
The main difficulty in practice is that \emph{we do not know $\opt[\M]$ in advance}.

Thus, the algorithms exploiting this smooth substructure %
require two %
ingredients: 
\begin{itemize}
  \item[i)] a mechanism to identify the optimal manifold;
  \item[ii)] an efficient method to minimize $\funcomp$ restricted to this manifold.
\end{itemize}

\smallskip

For general convex functions, the algorithm of \cite{mifflin2005algorithm} 
mixes a proximal bundle iteration (as a heuristic for 
identification) and
a so-called $\mathcal{U}$-Newton iteration (which interprets as an SQP step; see \cite[Sec.\;5]{miller2005newton}).
The obtained superlinear rate hinges on the %
identification of the optimal manifold.

For max-of-smooth functions~\eqref{example:maxofsmooth}, the paper \cite{womersleyAlgorithmCompositeNonsmooth1986} pioneered the idea of seeking the optimal manifold and using it to make second-order steps. Their identification heuristic uses the indices of the maximal function along a descent direction. Recently, \cite{lewisSimpleNewtonMethod2019, hanSurveyDescentMultipoint2021} investigate a related setting and propose bundle-like algorithms incorporating high-order information that converge (super)linearly on max-of-smooth functions when %
the optimal manifold is known.

For the maximum eigenvalue of a parametrized matrix \eqref{example:eigmax}, a specific version of the $\mathcal{U}$-Newton method discussed above is studied by \cite{noll2005spectral}. Again, the identification mechanism is a heuristic determining the multiplicity of the maximal eigenvalue and the optimization step is an SQP iteration.

None of these methods guarantee identification of the optimal manifold: they either assume that the optimal manifold is known in advance, or rely on heuristics for identification. Here, we aim at further harnessing the smooth substructure of $\funcomp=\funns\circ\mapping$ to have \emph{guaranteed local identification} of the optimal manifold and then \emph{guaranteed quadratic convergence} when using SQP iterations.

\subsection{Contributions and outline}
We propose a \GBedit{local} second-order algorithm for solving the nonsmooth composite problem \eqref{eq:minF} that identifies
the optimal manifold of non-differentiability. The two main ingredients of our algorithm are the following:
\begin{itemize}
  \item[i)]  we use the explicit proximal operator of $\funns$ with chosen stepsizes to provide a guaranteed identification procedure;
  \item[ii)] for a candidate manifold $\M$, we make an SQP iteration minimizing a smooth extension of $\funcomp\vert_{\M}$ subject to the constraint of belonging to $\M$.
\end{itemize}

\smallskip
The fact \GBedit{that} proximal-based operators have identification properties around minimizers is well-known: the proximal operator\;\cite{hare2004identifying,daniilidis2006geometrical}, the proximal gradient operator\;\cite{bareilles2020newton}, approximate variable-metric proximal gradient operators\;\cite{leeAcceleratingInexactSuccessive2021}, and prox-linear operators\;\cite{lewis2016proximal} locally identify the optimal manifold under some natural geometrical assumptions.
Here, we only have access to the proximity operator of $\funns$, and in order to exploit the structure it provides, we face the double challenge of, first, identifying the smooth structure around a point which is not a minimizer for $\funns$, and, second, deducing the corresponding structure of $\funcomp = \funns\circ\mapping$.
Thus, our main technical contribution is to establish that $\prox_{\step \funns}$ maps a point $\vxinter$ close to $\mapping(\opt)$ to $\mapping(\opt[\M])$.
The step $\step$ should be carefully chosen, in particular larger than the distance of $\vxinter$ to $\mapping(\opt)$.
Mathematically, we study the range of steps for which the curve $\step \mapsto \prox_{\step \funns}(\vxinter)$ belongs to $\mapping(\opt[\M])$.
This analysis shows connections with recent works in nonsmooth analysis, such as the modulus of identifiability appearing in \cite{lewis2022identifiability}.

We combine this new identification result with standard SQP-steps to propose \GBreplace{an}{a local} algorithm for minimizing the composite function $\funcomp$.
We pay a special attention to prevent the quadratic convergence of SQP from jeopardizing identification: we prove that, for a well-chosen stepsize policy, the method identifies the optimal structure and \GBreplace{locally }{}converges quadratically.
We illustrate numerically these properties on problems of the form \eqref{example:maxofsmooth} and \eqref{example:eigmax}.

The outline of the remainder of the paper is as follows. First, in \Cref{sec:asm}, we introduce the technical tools to describe the manifold identification brought by proximity operators (including prox-regularity and partial smoothness). Furthermore, we lay out two technical properties needed for proximal identification in the composite setting.
In \Cref{sec:identif-prox}, we show our main result consisting in a description of a stepsize range for which the proximity operator of $\funns$ identifies the optimal manifold locally around a minimizer.
In \Cref{sec:method} we detail the proposed method combining SQP-steps and proximal identification steps. %
Finally, we present in \Cref{sec:numexps} numerical illustrations of our method and of the identification result.

\subsection{Notations}

\GBedit{
  Given a point $z$ in $\bbR^{p}$, we denote by $\N_{z}$ a neighborhood of $z$ in $\bbR^{p}$.
  We reserve the names based on $\vx$ for points in the input space $\bbR^{n}$, and on $\vxinter$ for points in the intermediate space $\bbR^{m}$.
  We denote the differential of a smooth mapping $\mapping$ at point $\vx$ by $\D \mapping(\vx)$, and its Jacobian matrix by $\Jac_{\mapping}(\vx)$.
  The conjugate of the linear operator $A$ is denoted by $A^{*}$.
  The minimizer is denoted by $\opt[\vx]$.
}

\section{Setting and assumptions}
\label{sec:asm}

Let us start by representing schematically the type of functions we consider:
\begin{align}
  \inputSpace \xrightarrow[\text{smooth mapping}]{\mapping} \Im(\mapping) \subset \interSpace  \xrightarrow[\text{nonsmooth function}]{\funns} \bbR\cup\{+ \infty\}.
\end{align}
Throughout the paper, we denote by $\vx$ points in the \emph{input space} $\inputSpace$ and by $\vxinter$ points in the \emph{intermediate space} $\interSpace$.

In all the results presented in this paper, we make the following assumption that describes the minimal global properties on $\funns$ and $\mapping$ to conduct our reasoning.
\begin{assumption}\label{asm:blanket}
  The mapping $\mapping:\inputSpace\to\interSpace$ is $\mathcal{C}^2$, the function $\funns:\interSpace\to\bbR\cup \{+\infty\}$ is proper and lower semi-continuous.
\end{assumption}

We work with the set of \emph{(general) subgradients} (see~\cite[Def.\;8.3]{rockafellar2009variational}), defined at a point $\crit[\vxinter]$ where $\funns(\crit[\vxinter])$ is finite as:
\begin{align}
    \partial \funns(\crit[\vxinter]) \defeq \left\{ \lim_r v_r : v_r\in\widehat\partial \funns(\vxinter_r), \vxinter_r\to\crit[\vxinter], \funns(\vxinter_r)\to \funns(\crit[\vxinter]) \right\},
\end{align}
where $\widehat\partial \funns(\crit[\vxinter])$ denotes the set of \emph{regular (or Fréchet) subgradients}, defined as
\begin{align}
    \widehat\partial \funns(\crit[\vxinter]) \defeq \left\{ v : \funns(\vxinter) \ge \funns(\crit[\vxinter]) + \langle v, \vxinter-\crit[\vxinter] \rangle + o(\|\vxinter-\crit[\vxinter]\|) \text{ for all } \vxinter\in\interSpace \right\}.
\end{align}
These two subdifferentials match if (and only if) $\funns$ is (Clarke) regular at $\crit[\vxinter]$. Closed convex functions are regular everywhere and these subdifferentials match the usual convex subdifferential (see~\cite[Chap.\;8.11-12]{rockafellar2009variational} for details).

In the remainder of this section, we provide quick recalls and definitions about the two important objects of our analysis: the proximity operator in \Cref{sec:defsnsfun} and the structure manifolds in \Cref{sec:structmanifolds}. We illustrate them on our running examples \eqref{example:maxofsmooth} and \eqref{example:eigmax}.

\subsection{Proximity operator}%
\label{sec:defsnsfun}

The proximity operator of a function $\funns$ with step $\step > 0$ at $\vxinter\in\interSpace$ is defined as the set-valued mapping
\begin{align}
    \prox_{\step \funns}(\vxinter) \defeq \argmin_{u\in\interSpace} \left\{ \funns(u) + \frac{1}{2\step}\|u-\vxinter\|^2\right\}.
\end{align}
This operator is well-defined when $\funns$ is prox-regular and prox-bounded; see\;\eg\cite[13.37]{rockafellar2009variational}. %
We quickly introduce these two notions %
and recall a result on the uniqueness and characterization of the prox operator, which is %
important in our developments.

A function $\funns$ is \emph{prox-regular} at a point $\crit[\vxinter]$ for a subgradient $\bar{v} \in \partial \funns(\crit[\vxinter])$ if $\funns$ is finite, locally lower semi-continuous at $\crit[\vxinter]$, and there exists $r>0$ and $\epsilon>0$ such that
\begin{align}\label{eq:proxregdef}
  \funns(\vxinter') \ge \funns(\vxinter) + \langle v, \vxinter'-\vxinter\rangle - \frac{r}{2}\|\vxinter'-\vxinter\|^2
\end{align}
whenever $v \in \partial \funns(\vxinter)$, $\|\vxinter-\crit[\vxinter]\|<\epsilon$, $\|\vxinter'-\crit[\vxinter]\|<\epsilon$, $\|v-\bar{v}\|<\epsilon$, and $\funns(\vxinter)<\funns(\crit[\vxinter])+\epsilon$.
When this holds for all $\bar{v} \in\partial \funns(\crit[\vxinter])$, we say that $\funns$ is prox-regular at $\crit[\vxinter]$~\cite[Def. 13.27]{rockafellar2009variational}.

A function $\funns$ is \emph{prox-bounded} if there exists $R\ge 0$ such that the function $\funns + \frac{R}{2}\|\cdot\|^{2}$ is bounded below.
The corresponding \emph{threshold} (of prox-boundedness) is the smallest $r_{pb}\ge 0$ such that $\funns + \frac{R}{2}\|\cdot\|^{2}$ is bounded below for all $R>r_{pb}$.
In this case, $\funns + \frac{R}{2}\|\cdot - \crit[\vxinter]\|^{2}$ is bounded below for any $\crit[\vxinter]$ and $R > r_{pb}$~\cite[Def.\;1.23, Th.\;1.25]{rockafellar2009variational}.

We can now recall a relevant result on the characterization of proximal points.
\begin{proposition}[{\cite[Th.\;1]{hare2009computing}}]\label{prop:wellbehavedprox}
  Suppose that the function $\funns$ is prox-regular at $\crit[\vxinter]$ for $\bar{v} \in \partial \funns(\crit[\vxinter])$ with parameter $r_{pr}$, and prox-bounded with threshold $r_{pb}$.
  Then, for any $\step < \min(r_{pr}^{-1}, r_{pb}^{-1})$ and all $\vy$ near $\crit[\vxinter] + \step \bar{v}$, the proximal operator is:
  \begin{itemize}
    \item single-valued and locally Lipschitz continuous;
    \item uniquely determined by the relation
          \begin{align}
            p = \prox_{\step \funns}(\vy) \Leftrightarrow \vy - p \in \step \partial \funns(p).
          \end{align}
  \end{itemize}
\end{proposition}

In addition to its existence and characterization provided by the result above, the proximity operator has a closed-form expression in our running examples.

\begin{example}[Maximum]
  The subdifferential of $\funns(\vxinter) = \max(\vxinter_{1}, \ldots, \vxinter_{\ndimalt})$ is
  \begin{align}
    \partial \max (\vy) = \Conv \left\{ e_i : \vy_i = \max(\vy) \right\},
  \end{align}
  where $e_i $ is the $i$-th element of the Cartesian basis of $\interSpace$.
  \GBreplace{This}{The $\max$} function is convex, thus globally prox-regular and prox-bounded everywhere (with parameters $0$).
  Its proximity operator is given (coordinate-wise) by
  \begin{align}
    \left[ \prox_{\step \max}(\vy) \right]_i = \begin{cases}
      s & \text{if } \vy_i > s \\
      \vy_i & \text{else}
    \end{cases}
  \end{align}
  where $s$ is the unique real number such that $\sum_{\{i : \vy_i>s\}} (\vy_i - s) =
  \step$.
\end{example}

\begin{example}[Maximum eigenvalue]
  Denote the eigenvalue decomposition of a point $\vxinter\in\Sym$ as $\vxinter = E\Diag(\lambda) E^\top$, where $\lambda\in\bbR^{m}$ is a vector with decreasing entries and $E\in\bbR^{m \times m}$ an orthogonal matrix.
  The subdifferential of the maximum eigenvalue at $\vxinter$ writes~\cite[Ex.~3.6]{lewis2002active}
  \begin{align}
    \partial \lammax(\vxinter) = \{E_{1:r}ZE_{1:r}^{\top}, Z\in\Sym[r], Z\succeq 0, \trace Z=1\}
  \end{align}
  where $r$ is the multiplicity of the maximum eigenvalue of $\vxinter$.
  \GBreplace{This}{The $\lammax$} function is convex, thus prox-regular and prox-bounded (with parameters $0$).
  Its proximity operator can be expressed using the one of the $\max$ function as
  \begin{align}
    \prox_{\step \lammax}(\vy) = E \Diag(\prox_{\step \max{}}(\lambda)) E^\top.
  \end{align}
\end{example}

\subsection{Structure manifolds}%
\label{sec:structmanifolds}

We now specify the notion of structure manifold in relation with a nonsmooth function $\funns$.

A subset $\M$ of $\RR^n$ is said to be a \emph{$p$-dimensional $\C^2$-submanifold} of $\RR^n$ around $\crit[\vx] \in \M$ if there exists a $\C^2$ \emph{manifold-defining map} $\maneq : \RR^n \to \RR^{n-p}$ with a surjective derivative at $\crit[\vx] \in \M$ that satisfies for all $\vx$ close enough to $\crit[\vx]$: $\vx \in\M \Leftrightarrow \maneq(\vx) = 0$.
We define the tangent and normal spaces at a point $\vx\in\M$ as follows:
\begin{align}
    \tangent{\vx}{\M} = \ker \D\maneq (\vx) \qquad \normal{\vx}{\M} = \Im\; \D\maneq (\vx)^*
\end{align}

The important notion of \emph{structure manifolds of $\funns$} can be defined as a manifold $\Minter$ where $\funns$ is nondifferentiable. More precisely, at a point  $\crit[\vxinter]\in\Minter$, we require $\funns$ to be prox-regular and \emph{partly smooth}. 
This property of ($\C^2$-)partial smoothness is verified at a point $\crit[\vxinter]$ for a function $\funns$ relatively to a set $\Minter$ containing $\crit[\vxinter]$ if $\Minter$ is a $\C^2$ manifold around $\crit[\vxinter]$ and if
\begin{itemize}
  \item (smoothness) the restriction of $\funns$ to $\Minter$ is a $\C^2$ function near $\crit[\vxinter]$;
  \item (regularity) $\funns$ is (Clarke) regular at all points $\vxinter\in\Minter$ near $\crit[\vxinter]$, with $\partial \funns(\vxinter)\neq\emptyset$;
  \item (sharpness) the affine span of $\partial \funns(\crit[\vxinter])$ is a translate of $\normal{\crit[\vxinter]}{\Minter}$;
  \item (sub-continuity) the mapping $\partial \funns$ restricted to $\Minter$ is continuous at~$\crit[\vxinter]$.
\end{itemize}

The concept of partial smoothness, introduced in~\cite{lewis2002active}, captures (locally) well-behaved nonsmoothness by requiring $\funns$ to be smooth along a manifold and non-smooth across it.
In addition, the prox-regularity of $\funns$ ensures \GBreplace{unicity}{uniqueness} of the structure manifold near $\crit[\vxinter]$~\cite[Corollary 4.2, Example 7.1]{hare2004identifying}.
To highlight the relation between the manifold and the function $\funns$, we use the notation $\Minter$ for the structure manifold related to $\funns$.

\begin{example}\label{ex:maxofsmoothstruct}
  The structure manifolds of $\max$ are
  \begin{align}
    \MI = \{ \vy \in \interSpace : \vy_i = \max(\vy) \text{ for } i\in I \},
  \end{align}
  where $I \subset \{1, \ldots, m\}$.
  A smooth manifold-defining map for $\MI$ is $\maneq:\interSpace\to \bbR^{|I|-1}$ such that $\maneq(\vxinter)_{l} = \vxinter_{i_{l}} - \vxinter_{i_{|I|}}$, where $|I|$ denotes the size of $I$ and $i_{l}$ the $l$-th element of $I$ (with some ordering).
  As required, this map is surjective.
  At any point $\vxinter \in \interSpace$, the maximum is partly smooth relative to $\MI$, where $I = \{i : \vxinter_{i} = \max(\vxinter)\}$.
\end{example}

\begin{example}\label{ex:eigmaxstruct}
  The structure manifolds of $\lammax$ in $\Sym$ consist of all matrices having a largest eigenvalue with fixed multiplicity $\mult$:
  \begin{align}
    \Mr = \{\vxinter \in \Sym : \lambda_1(\vxinter) = \dots = \lambda_r(\vy)\}.
  \end{align}
  A manifold-defining map of $ \Mr $ is described in~\cite{shapiroEigenvalueOptimization1995a} and $\lammax$ is partly smooth relative to $\Mr$ at any point $\vxinter \in \Mr$. 
\end{example}

In view of the expression of the proximity operators in our examples, their output naturally lie on the structure manifolds described above. More precisely, $\prox_{\step \max}(\vxinter)$ belongs to the structure manifold $\MI$, where $I$ collects the indices of the $k$ largest entries of $\vxinter$ and $k$ grows as $\step$ increases. Similarly, $\prox_{\step \lammax}(y)$ belongs to the structure manifold $\Mr$, where $r$ increases as $\step$ does. This observation is at the core of the ability of proximal operators to identify neighboring structure manifolds.

\subsection{Structure Identification}%
\label{subsec:addasm}

It is well-known that the proximity operator identifies structure locally around critical points (see \eg~\cite[Th.\;28]{daniilidis2006geometrical}): all points near a minimizer are mapped to the manifold containing the minimizer. Furthermore, this structure is revealed during the computation of the operator.\footnote{Computing exactly the \emph{structure} of the output point of the operator, as can be done for the prox, is opposed to merely observing the structure of the output after its computation. This last option is not desirable in our opinion as it entails delicate numerical questions such as testing equality between reals for the maximum, or computing the multiplicity of the maximal eigenvalue of a matrix.}

In the situation we consider, the proximity operator of $\funcomp$ cannot be explicitly computed. However, $\prox_{\step \funns}$ is available and can provide some structure in the intermediate space $\interSpace$ that we would like to exploit. To do so, we introduce two properties (holding on our two running examples), that will allow us to retrieve the structural information in the intermediate space near points that are not minimizers of $\funns$.

The first property holds at point $\crit[\vxinter]\in\Minter$ if the nonsmooth function $\funns$ strictly increases on all directions on which it is nonsmooth.
\begin{property}[Normal ascent]\label{prop:normalascent}
  A function $\funns$ satisfies the \emph{normal ascent} property at point $\crit[\vxinter]$ if $0$ lies in the relative interior of the projection of $\partial \funns (\crit[\vxinter])$ on the normal space at $\crit[\vxinter]$, that is:
    $0 \in \ri \proj_{\normal{\crit[\vxinter]}{\Minter}} \partial \funns(\crit[\vxinter])$.
\end{property}

\begin{remark}[Positive directional derivative]
  In a ``nice'' setting where $\funns$ is Lipschitz continuous and regular at $\crit[\vxinter]$, \Cref{prop:normalascent} implies that the directional derivative of $\funns$ along any normal direction $d\in \normal{\crit[\vxinter]}{\Minter}$ is positive.
  Indeed, in that case one-sided directional derivatives are well-defined \cite[p.\;358, Th.\;9.16]{rockafellar2009variational}, and the derivative along direction $w$ equals $\max_{v\in\partial \funns(\GBtwoedit{\crit[\vxinter]})} \langle v, w \rangle$.
  Along a normal direction $d \in \normal{\GBtwoedit{\crit[\vxinter]}}{\Minter}$, by partial smoothness the directional derivative writes $\max_{\vn\in \proj_{\normal{\GBtwoedit{\crit[\vxinter]}}{\Minter}} (\partial \funns( \GBtwoedit{\crit[\vxinter]}))} \langle \vn, d \rangle$.
    \Cref{prop:normalascent} ensures the existence of $\alpha >0$ such that $\alpha d \in \proj_{\normal{\GBtwoedit{\crit[\vxinter]}}{\Minter}} (\partial \funns(\GBtwoedit{\crit[\vxinter]}))$, making the derivative positive.

\end{remark}

\FIedit{Let us briefly discuss that, even if \Cref{prop:normalascent} may look strong, in practice it is not.
  For a given nonsmooth function $\funcomp$ which can be decomposed as $\funcomp = \funns\circ \mapping$, \Cref{prop:normalascent} may not hold for $\funns$ at $\mapping(\opt)$ for a minimizer $\opt$.
  Nevertheless, the property often holds for a different decomposition $\funcomp = \tilde{\funns} \circ \tilde{\mapping}$.
  We give two examples where changing the decomposition of $\funcomp$ ensures that \Cref{prop:normalascent} holds at minimizers.

  \begin{example}[Normal ascent for regularized-type problem]
    Consider the minimization of $F(x) = f(x) + r(x)$, where $f(x) = \frac{3}{2}x$ and $r(x) = |x| - \frac{1}{2}x$, whose minimizer is $\opt=0$.
    This writes as a composite problem by setting $\mapping(x) = (f(x), x)$ and $\funns(y) = y_{1} + r(y_{2})$. We note first that 
    \Cref{prop:normalascent} does not hold for $\funns$ at $\mapping(\opt)$, the structure manifold of $g$ at $\mapping(\opt)$ being $\Minter = \bbR \times \{0\}$.
    However the function also writes $\funcomp(x) = \tilde{f}(x) + \tilde{r}(x)$, with $\tilde{f}(x) = x$ and $\tilde{r}(x) = |x|$.
    Letting similarly $\tilde{\mapping}(x) = (\tilde{f}(x), x)$ and $\tilde{g}(y) = y_{1} + \tilde{r}(y_{2})$, we now get that \Cref{prop:normalascent} holds for $\tilde{g}$ at $\tilde{\mapping}(\opt)$.
  \end{example}

  \begin{figure}[t]
    \centering
    \includetikzfig[.35]{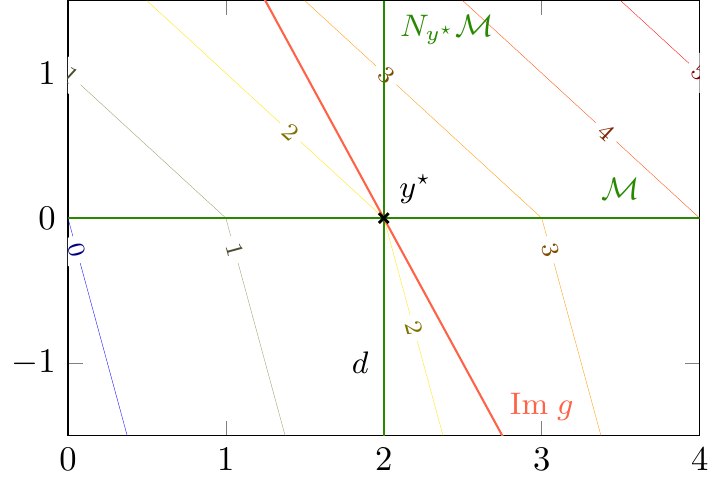}
    \caption{
      \GBedit{
      Illustration of the level-curves of function $g$ in \cref{ex:normalascentcompo}, along with the image of $\mapping$ and the tangent and normal spaces to $\Minter$ at the minimizer.}%
      \label{fig:normalgrowthex}
      \vspace*{-4ex}
    }
  \end{figure}
  \begin{example}[Normal ascent property for composite problems]%
    \label{ex:normalascentcompo}
    Consider the minimization of $\funcomp = \funns\circ \mapping$, with
    \begin{align}
      \funns(\vy) =
      \begin{cases}
        \vy_{1}+\vy_{2} &\text{if } \vy_{1}>0 \\
        \vy_{1} + 0.25\,\vy_{2} &\text{else}
      \end{cases}, \qquad \mapping(\vx) = \begin{pmatrix}
                                       2 - \vx \\ 2\vx
                                     \end{pmatrix}.
    \end{align}
    The minimizer is $\opt[\vx] = 0$, since $\funns$ is strictly increasing at all $\vxinter\in \Im (\mapping)$ near $\opt[\vxinter] = \mapping(\opt[\vx])$; see \cref{fig:normalgrowthex}.
    However the normal ascent property does not hold at $\opt$: $\funns$ is decreasing at $\opt[\vxinter]$ along the normal direction $(0 ; -1)$.

    The composite function boils down to $\funcomp(\vx) = 2+\max(\vx, -0.5\vx) = \tilde{\funns}\circ \tilde{\mapping}(\vx)$, where $\tilde{\funns}(\vxinter) = 2+\max(\vxinter)$ and $\tilde{\mapping}(\vx) = (\vx, -0.5\vx)$.
    With this decomposition, $\tilde{\funns}$ does satisfy the normal ascent property at $\opt[\vx]$.
  \end{example}
}

The second property is more technical and controls the velocity of a curve  on the manifold $\Minter$.

\begin{property}[Curve property]\label{prop:curve}
  A function $\funns$ partly smooth at $\crit[\vxinter]$ relative to $\Minter$ satisfies the \emph{curve property} at $\crit[\vxinter]$ when there exists a neighborhood $\N_{\crit[\vxinter]}$ of $\crit[\vxinter]$ and $\Tcurve>0$ such that any smooth application $\cm:\N_{\crit[\vxinter]}\times [0, \Tcurve]\to\Minter$ such that $\cm(\vxinter,0) = \projM[\Minter](\vxinter)$, $\frac{\dd}{\dd t}e(\vxinter, t)\vert_{t=0} = -\grad \funns(\projM[\Minter](\vxinter))$ satisfies
  \begin{align}
    \| \proj_{\normal{\cm(\vxinter, t)}{\Minter}}(\cm(\vxinter, t) - \vxinter) \| \le \distM[\Minter](\vxinter) + \constcurve \; t^{2} \quad \text{ for all } \vxinter\in\N_{\crit[\vxinter]}, t \in [0, \Tcurve],
  \end{align}
  where $\distM[\Minter](\vxinter)\defeq \|\vxinter - \projM[\Minter](\vxinter)\|$ is the distance between $\Minter$ and $\vxinter$, and $\grad \funns(p) \in \tangent{p}{\Minter}$ denotes the Riemannian gradient of $\funns$ obtained as $\proj_{\tangent{p}{\Minter}}( \partial \funns(p))$.
\end{property}

The idea behind this property is to ensure that the differential of the \GBreplace{projection of the (time dependent)}{(time dependent) projection on the} normal space is (uniformly) negligible at time $0$.
Note that for affine spaces, we trivially have $ \| \proj_{\normal{\cm(\vxinter, t)}{\Minter}}(\vxinter - \cm(\vxinter, t)) \| = \distM[\Minter](\vy) $ for all $t$ near $0$: the normal spaces are equal at all points of the manifold.

These two properties are satisfied at any structured point for the two nonsmooth functions $\max$ and $\lammax$ of our running examples as detailed in the following lemma. The proofs for the two functions are rather direct but require precise technical descriptions; we defer them to \Cref{sec:addasmproof}.

\begin{lemma}\label{lemma:props}
  Consider either:
  \begin{itemize}
    \item $\funns = \max$, $\crit[\vxinter] \in \interSpace$, and the structure manifold $\MI$ (of~\Cref{ex:maxofsmoothstruct});
    \item $\funns = \lammax$, $\crit[\vxinter] \in \Sym$, and the structure manifold $\Mr$ (of~\Cref{ex:eigmaxstruct}).
  \end{itemize} 
  Then, \Cref{prop:normalascent,prop:curve} hold at $\crit[\vxinter]$.
\end{lemma}

Finally, the structure provided by $\prox_{\step \funns}$ lies in the intermediate space $\interSpace$, while the optimization variable lives in $\inputSpace$.
In order to transfer the structure information to the input space, we will also require the smooth map $\mapping:\inputSpace\to\interSpace$ to be transversal to $\Minter \subset\interSpace$ at some point $\crit\in\inputSpace$, which holds when $\Minter$ is a manifold around $\mapping(\crit)$ and the following (equivalent) conditions hold:
\begin{align}%
  \label{eq:transversality}
  \normal{\mapping(\crit)}{\Minter} \cap \ker(\D \mapping(\crit)^{*}) = \{0\} \quad \text{ or } \quad \tangent{\mapping(\crit[\vx])}{\Minter} + \Im\; \D \mapping(\crit[\vx]) = \interSpace.
\end{align}
In that case, the set $\mapping^{-1}(\Minter)$ is a submanifold of $\inputSpace$~\cite[Th. 6.30]{leeIntroductionSmoothManifolds2003}, whose normal space has the same dimension as the one of $\Minter$. Furthermore, we have~\cite[Ex.~6-10]{leeIntroductionSmoothManifolds2003}
\begin{align}%
  \label{eq:normalspaceintertoinput}
  \normal{\vx}{\mapping^{-1}(\Minter)} = \D \mapping(\vx)^{*} \normal{\mapping(\vx)}{\Minter}  ~ \text{ and } ~ \tangent{\vx}{\mapping^{-1}(\Minter)} = \D \mapping(\vx)^{-1} \tangent{\mapping(\vx)}{\Minter}.
\end{align}

\section{Collecting structure with the proximity operator}\label{sec:identif-prox}

We show in this section how to exactly detect the optimal structure manifold of the composite function $\funcomp = \funns \circ \mapping$ around a point $\crit[\vx]$ using the proximity operator of $\funns$. %

In our nonconvex and nonsmooth setting, we seek only structured points which satisfy certain assumptions summarized in our definition of a \emph{qualified point}.

\begin{definition}[Qualified points]\label{def:qualifiedpoints}
  A point $\crit[\vx]\in\inputSpace$ is \emph{qualified} relative to a decomposition $(\funns, \mapping)$ of $\funcomp$ and manifold $\Minter$ if
  \begin{enumerate}
    \item $\funns$ is prox-bounded and prox-regular at $\mapping(\crit[\vx])$;
    \item $\funns$ is partly smooth at $\mapping(\crit)$ relative to $\Minter$;
    \item $\mapping$ is transversal to $\Minter$ at $\crit$;
    \item $\funns$ satisfies~\Cref{prop:normalascent,prop:curve} at point $\mapping(\crit[\vx])$.
  \end{enumerate}
\end{definition}

Three of these assumptions constrain only the nonsmooth function $\funns$ and are easily verifiable in practice.
Only the transversality condition limits the range of acceptable smooth mappings; see \eg~\cite[Sec.\;4]{lewis2002active}.
For such \emph{qualified} points, we get two useful properties: first, $\funcomp$ is partly smooth at $\crit[\vx]$ relative to \GBedit{the manifold $\M$, locally defined as $\M \defeq c^{-1}(\Minter) \ni \crit[\vx]$ by the chain rule of~\cite[Th. 4.2]{lewis2002active}}, and second, the operator $\prox_{\step \funns}$ is single-valued, locally Lipschitz, and defined by its optimality condition near $\mapping(\crit[\vx])$.

\subsection{Main result\GBedit{: $\prox_{\step \funns}\circ \mapping$ as a structure detector}}

We show in the following theorem that if $\vx$ is near a qualified point of $\funcomp$ with structure $\M$, then $\prox_{\step \funns}(\mapping(\vx))$ will output a point on $\Minter = \mapping(\M)$, the structure manifold of $\funns$ corresponding to $\M$ (in the intermediate space).
Our theorem provides precise conditions on $\vx$ and $\step$ that guarantee this structure identification and forms the main theoretical contribution of the paper.
We illustrate this behavior in \GBedit{\Cref{fig:3funs,fig:3funsinter}.}

The position of this result with respect to the literature is discussed right after in~\Cref{rmk:positionlit}, and the proof is given in the following~\Cref{sec:proofidentif}, in a succession of technical lemmas.
We stress that we give guarantees on the \emph{structure} to which the point $\prox_{\step \funns}(\mapping(\vx))$ belongs, rather than on the point itself.

\begin{theorem}\label{th:proxintspace}
  Consider a function $\funcomp = \funns\circ \mapping$ and a point $\crit[\vx]$.
  Assume that $\crit[\vx]$ is \emph{qualified} relative to a manifold $\Minter \subset \interSpace$. Then, there exists a neighborhood $\N_{\crit}$ of $\crit$ and a constant $\stepUpBnd$ such that, for all $\vx \in \N_{\crit}$,
  \begin{align}
    \prox_{\step \funns}(\mapping(\vx)) \in \Minter \text{ for all } \step \in [\stepLowBnd(\distM(\vx)), \stepUpBnd],
  \end{align}
  where $\distM(\vx)$ denotes the distance from $\vx$ to the manifold $\M$ and $\stepLowBnd$ is defined as
  \begin{align}
    \stepLowBnd(t) = \frac{\cri}{2\constcurve} \left( 1- \sqrt{1 - \frac{4\constcurve\cmap t}{\cri^{2}}} \right) = \frac{\cmap}{\cri}t + \frac{\constcurve\cmap^{2}}{\cri^{3}}t^{2} + o(t^{2}),
  \end{align}
  with $\cri$, $\cmap$, and $\constcurve$ (of~\Cref{prop:curve}) positive constants.

  In particular, there exists $L>0$, $\epsilon>0$ such that
  \begin{align}
    \|\vx - \opt\| \le \epsilon \text{ and } L\|\vx - \opt\| \le \step \le \stepUpBnd \Longrightarrow \prox_{\step \funns}(\mapping(\vx)) \in \Minter.
  \end{align}
\end{theorem}
Note that~\Cref{prop:curve} is only used %
to compute explicitly an interval of\;$\step$ guaranteed to provide the correct structure; the existence of that interval holds independently.

\begin{remark}[Relation with existing results]\label{rmk:positionlit}
  The difference between~\Cref{th:proxintspace} and existing results lies in two aspects.
  First, the identification properties of the proximal operator~\cite[Th.\;28]{daniilidis2006geometrical}, the proximal-gradient operator~\cite[Th.\;3.1]{bareilles2020newton}, or even approximate prox-gradient operators \cite{leeAcceleratingInexactSuccessive2021} give structure information directly in the input space (even in abstract algorithmic frameworks~\cite[Th.\;4]{hare2004identifying} \GBedit{or~\cite[Th. 4.10]{lewisPartialSmoothnessTilt2013}}).
  In the composite case, the proximity operator reveals structure in the intermediate space only, and extra work is required to bring it back to the input space.

  Second, most existing results investigate identification properties near minimizers, and not just arbitrary points \GBreplace{(a notable exception is \cite{bareilles2020newton} in a different context)}{(two notable exceptions give results near arbitrary structured points:~\cite{lewisPartialSmoothnessTilt2013} for an abstract algorithmic framework, and \cite{bareilles2020newton} for the proximal gradient)}.
  Here, we evaluate $\prox_{\step \funns}$ near $\mapping(\crit[\vx])$, a point without any specific properties (even if $\crit[\vx]$ is a local minimizer).
  This is why we need \Cref{prop:normalascent} to guarantee identification in the intermediate space, and bring the structure information to the input space.
\end{remark}

\begin{figure}[p]
  \centering
  \includetikzfig[0.48]{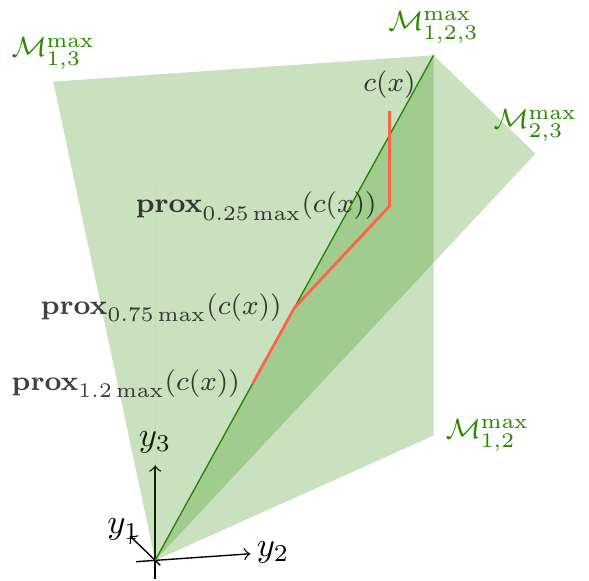}
  \caption{\GBedit{
      Illustration of the main result in the intermediate space, on the function of \cref{fig:3funs}.
      The structure manifolds of $\max : \bbR^{3} \to \bbR$ are displayed as the three half-planes and the line in green.
      The red line illustrates the curve $\gamma \mapsto \prox_{\step \max}(c(x))$.
      When $\step< 0.25$, the curve does not lie on any structure manifold.
      For $\step \in [0.25, 0.75)$, the curve lies on the optimal manifold $\M_{{2, 3}}^{\max}$.
      For $\step \ge 0.75$, the curve lies on $\M_{{1, 2, 3}}^{\max}$.%
      \label{fig:3funsinter}
  }
}
\end{figure}

\begin{figure}[p]
  \centering
  \begin{subfigure}[t]{0.46\textwidth}
    \centering
    \includetikzfig{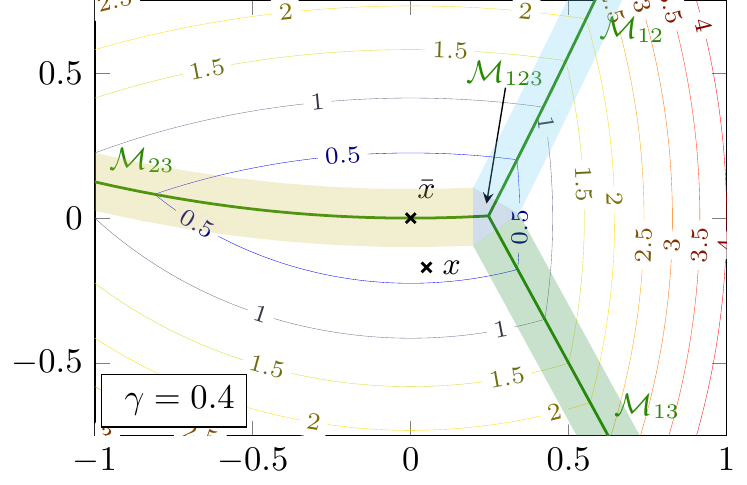}
  \end{subfigure}
  \hfill
  \begin{subfigure}[t]{0.46\textwidth}
    \centering
    \includetikzfig{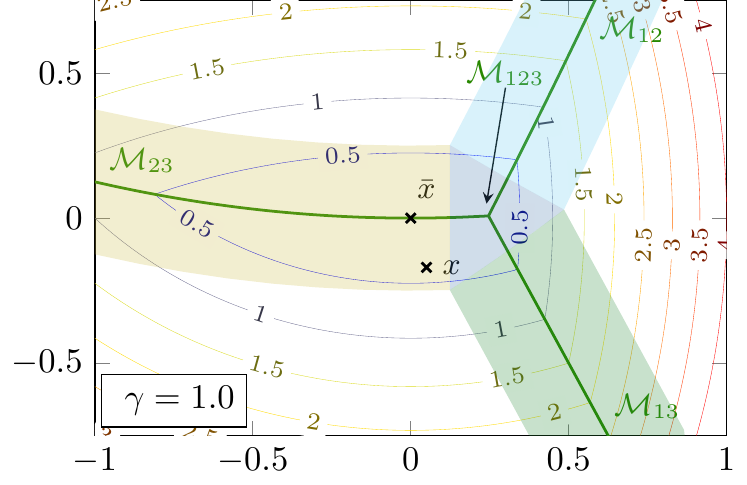}
  \end{subfigure}
  \vspace*{-2ex}
  \begin{subfigure}[a]{0.46\textwidth}
    \centering
    \includetikzfig{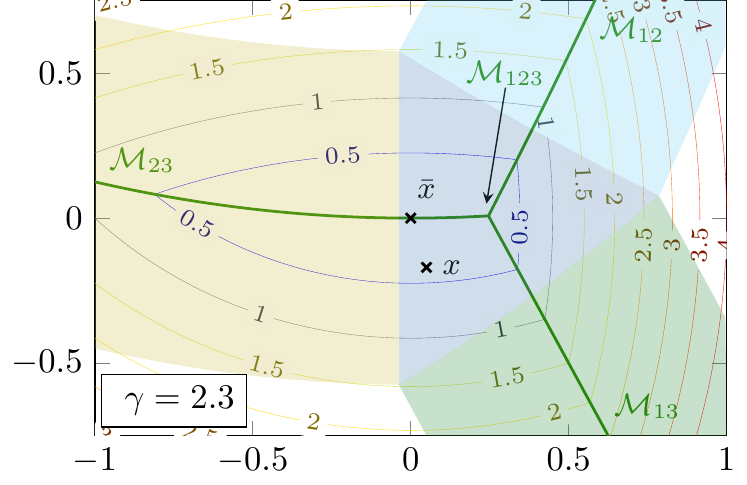}
  \end{subfigure}
  \hfill
  \begin{subfigure}[a]{0.46\textwidth}
    \centering
    \includetikzfig{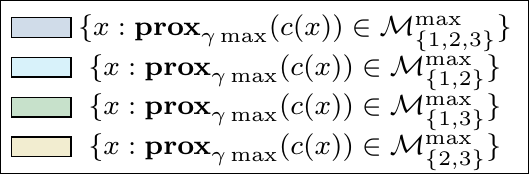}
  \end{subfigure}
  \caption{
    Illustration of the main result on a maximum of three quadratic functions, with $\crit \in \M^{\max}_{\{1, 2\}}$ and a point \GBreplace{$\tilde x$}{$\vx$} near $\crit$.
    The three figures show the areas where $\prox_{\step \funns}\circ \mapping$ detects manifolds for three stepsizes: $\step = 0.4$ (upper left), $\step = 1$ (upper right) and $\step = 2.3$ (lower left).
    We see on the upper left fig.\;that $\prox_{\step \funns}\circ \mapping$ detects no structure from \GBreplace{$\tilde x$}{$\vx$} because $\step$ is too small, and in \GBreplace{constrast}{contrast}, on the lower fig., that it wrongly detects too much structure ($\M_{\{1,2,3\}}^{\max}$) because $\step$ is too large.
    On the upper right fig., the optimal manifold is detected with $\step$ chosen in the right interval.%
    \label{fig:3funs}
  }
  \vspace*{-2ex}
\end{figure}

\GBedit{
  \begin{remark}[About prox-linear methods]
    Prox-linear methods are known to identify structure on composite problems~\cite{lewis2016proximal}.
    Specifically,~\cite[Th. 4.11]{lewis2016proximal} establishes that, after some finite time, an intermediate quantity \GBtwoedit{defined from the prox-linear subproblem \emph{exact} solution} belongs to the \GBtwoedit{optimal} structure manifold \GBtwoedit{$\opt[\Minter]$}.
    \GBtworeplace{It is then mentioned that}{In principle,} this information could be used to take efficient second-order steps to minimize $\funcomp$ along the identified manifold \GBtwoedit{$\opt[\M]$}.
    \GBtworeplace{Whether this can be done generically is unclear to us: checking that this quantity, obtained from the subproblem solution, belongs to a structure manifold is delicate.
    Though it is reasonable if the subproblem is solved with a suitable active-set method, it becomes delicate if only an approximation of the subproblem solution is available, using \eg interior point methods.
    In that case, the quantity will be somewhat close to the structure manifold, and one would have to resort to $\epsilon$-based tests.}{
    However, for generic composite problems, it may be difficult, first, to obtain an \emph{exact} solution of the prox-linear subproblems, and, second, to check if the ensuing quantity belongs to $\opt[\Minter]$.
    In contrast, the approach presented here only needs to compute an exact solution of the proximal operator of the simple nonsmooth function $\funns$.
  }
  \end{remark}

  \begin{remark}[\Cref{th:proxintspace} provides a structure identification tool]
    In contrast with the identification of prox-linear methods, \cref{th:proxintspace} provides a simple result for the detection of structure manifolds near any point $\vx\in\inputSpace$.
    We also underline that the bounds on the range of $\step$ that provide correct identification are surprisingly simple: the upper bound is constant and the lower bound is essentially \GBtworeplace{a linear function of to the distance to the manifold}{proportional to the distance to the manifold}.
    These simple and explicit bounds allow us to build a simple algorithm in the forthcoming \Cref{sec:method}.
  \end{remark}
}

\subsection{Proof of \Cref{th:proxintspace}}\label{sec:proofidentif}

\FIedit{
  The main difficulty of the proof is to build a suitable identification result for the nonsmooth function $\funns$.
  \Cref{th:proxintspace} (identification for $\funns\circ \mapping$) would then follow by taking into account the action of the smooth map $\mapping$.

  To derive an identification result on $\funns$, we have to give conditions on $\vxinter$ and $\step$ so that $p = \prox_{\step \funns}(\vxinter)$ lies on the considered manifold $\Minter$.
  Since $\funns$ is prox-regular and prox-bounded at point $\mapping(\crit[\vx])$, \cref{prop:wellbehavedprox} allows us to characterize this relation by its first-order optimality condition:
  \begin{align}
    \vxinter \in p + \step \partial \funns(p).
  \end{align}

  Whenever $p\in\Minter$ (which is what we want to show), this inclusion decomposes along $\tangent{p}{\Minter}$ and $\normal{p}{\Minter}$ as:
  \begin{align}
    \projT{p}{\Minter}(\vxinter - p) &= \step \grad \funns(p) \label{eq:expltanpart}\\
    \proj_{\normal{p}{\Minter}}(\vxinter - p) &\in \step \proj_{\normal{p}{\Minter}}  \partial \funns(p). \label{eq:explnormpart}
  \end{align}
  Thus we will show that for suitable $(\vxinter,\step)$, there is a unique $p$ that satisfies these two equations. We do so by considering the smooth  tangent component \cref{eq:expltanpart} first and then the nonsmooth  normal component \cref{eq:explnormpart} as follows:
  \begin{itemize}
    \item  We first show in \Cref{lem:smoothcurve} that for $\vxinter$ near $\crit[\vxinter]$ and $\step$ small, there exists a unique point $p=\cm(\vxinter, \step)$ on $\Minter$ that satisfies \cref{eq:expltanpart}, which depends smoothly on $\gamma$ and $\vxinter$. 
    This result is obtained by applying the implicit function theorem.
    \item Then, we prove in \Cref{lem:fullcurve} that $\cm(\vxinter, \step)$ also satisfies the second inclusion \cref{eq:explnormpart} if $\step$ belongs to the interval $[\stepLowBnd^{\funns}(\distM[\Minter](\vxinter)), \stepUpBnd^{\funns}]$. This result is a consequence of the application of some %
    variational analysis tools.
  \end{itemize}
  Putting these two results together, we obtain the existence and uniqueness of a point $p=\cm(\vxinter, \step) \in \Minter$ verifying both  \cref{eq:expltanpart} and \cref{eq:explnormpart} for all $\vxinter$ near $\crit[\vxinter]$ and $\step \in [\stepLowBnd^{\funns}(\distM[\Minter](\vxinter)), \stepUpBnd^{\funns}]$. By the first-order optimality condition presented above, this point is necessarily $\prox_{\step \funns}(\vxinter)$.

  Finally, this identification result in the intermediate space on $\funns$ is transferred back to the input space using transversality.
}

\subsubsection{\GBedit{Part 1: tangent optimality}}

We first show that, for $\vxinter$ near $\crit[\vxinter]$ and $\step$ small, there is a unique point $p$ \emph{on the manifold}  $\Minter$ that satisfies the tangent component of this optimality condition:
\begin{align}\label{eq:proxOCtangent}
  \tag{\ref{eq:expltanpart}}
  \projT{p}{\Minter}(\vxinter - p) = \step \grad \funns(p),
\end{align}
where $\grad \funns(p) \defeq \proj_{\tangent{p}{\Minter}} \partial (\funns(p))$ is unique by the sharpness property of partial smoothness, and matches the Riemannian gradient of $\funns$ on $\Minter$ (see~\cite[Sec. 7.7]{boumal2022intromanifolds}).
Such points $p$ are given by a smooth manifold-valued application $\cm(\vxinter, \step)$, the existence of which is guaranteed by the following lemma.

\begin{lemma}\label{lem:smoothcurve}
  Consider a function $\funns: \interSpace\to {\RR}\cup\{+\infty\}$, a point $\crit[\vxinter]\in\interSpace$, and a manifold $\Minter$ with $\funns$ partly smooth at $\crit[\vxinter]$ relative to $\Minter$.
  Then, there exists a smooth curve $\cm:\N_{\crit[\vxinter]}\times \N_{0} \to \M$ defined on a neighborhood of $(\crit[\vxinter], 0)$ in $\interSpace\times\bbR_{+}$ such that
  \begin{itemize}
    \item for all $y\in\N_{\crit[\vxinter]}$, $\cm(y, 0) = \projM[\Minter](y)$ and $\frac{\dd}{\dd \step}\cm(y, \step)\vert_{\step=0} = -\grad \funns(\projM[\Minter](y))$;
    \item for all $y\in\N_{\crit[\vxinter]}$, $\step\in\N_{0}$,~Eq.~\eqref{eq:proxOCtangent} is satisfied for $p=\cm(\vxinter, \step)$.
  \end{itemize}
\end{lemma}
\begin{proof}
  We define the mapping $\Phi:\interSpace \times \bbR \times \Minter\to\cup_{\vx\in\Minter}\tangent{\vx}{\Minter}$ as
  \begin{align}
    \Phi(\vxinter, \step, p) = \step \grad \funns(p) - \projT{p}{\Minter}(\vxinter - p)
  \end{align}
  and consider the equation $\Phi(\vxinter, \step, p) = 0$ near the point $(\crit[\vxinter], 0, \crit[\vxinter])$.
  Using the smoothness of $\funns$ on $\Minter$ given by partial smoothness, we have that this mapping is continuously differentiable on a neighborhood of $(\crit[\vxinter], 0, \crit[\vxinter])$.
  We see that its differential with respect to $p$ is $\D_{p} \Phi(\crit[\vxinter], 0, \crit[\vxinter]) = I$.
  Indeed, for $\eta \in \tangent{p}{\Minter}$,
  \begin{align}
    \D_{p} \Phi(\vxinter, \step, p)[\eta] = \step \Hess \funns(p)[\eta] + \eta - \D_{p'}\left(p'\mapsto\projT{p'}{\Minter}(\vxinter - p)\right)(p)[\eta].
  \end{align}
  At point $(\crit[\vxinter], 0, \crit[\vxinter])$, the first term vanishes, and the third term  writes
  \begin{align}
    \D_{p'}\left(p'\mapsto\projT{p'}{\Minter}(0)\right)(\crit[\vxinter])[\eta]
  \end{align}
  and vanishes as well as the differential of the null function $p'\mapsto\projT{p'}{\Minter}(0)$.
  Thus $\D_{p} \Phi(\crit[\vxinter], 0, \crit[\vxinter]) = I$ is invertible.
  The implicit functions theorem thus grants the existence of neighborhoods $\N_{\crit[\vxinter]}^{1}$, $\N_{0}^{2}$, $\N_{\crit[\vxinter]}^{3}$ of $\crit[\vxinter]$, $0$, $\crit[\vxinter]$ in $\interSpace$, $\RR$, $\Minter$ and a continuously differentiable function $\GBreplace{\c}{\cm}:\N_{\crit[\vxinter]}^{1} \times \N_{0}^{2}\to\N_{\crit[\vxinter]}^{3}$ such that, for any $(y, \step)\in\N_{\crit[\vxinter]}^{1} \times \N_{0}^{2}$,~\Cref{eq:proxOCtangent} is satisfied with $p=\cm(\vxinter, \step)$.
  For $y\in\N_{\crit[\vxinter]}^{1}$, $\cm(y, 0)$ satisfies $y - \cm(y, 0) \in \normal{\cm(y, 0)}{\Minter}$, which is the first-order optimality condition of $\cm(y, 0) = \projM[\Minter](y)$.
  Possibly reducing $\N_{\crit[\vxinter]}^{1}$ so that, for all $y\in\N_{\crit[\vxinter]}$ $\projM[\Minter](y)$ is well-defined and unique, the previous optimality condition is equivalent to $\cm(y, 0) = \projM[\Minter](y)$.
  Besides, differentiating $\Phi(y, \step, \cm(\vxinter, \step)) = 0$ relative to $\step$ at $\step=0$ yields
  \begin{align}
    \D_{\step} \cm(y, 0) &= - [\D_{p} \Phi(y, 0, \projM[\Minter](y))]^{-1} \D_{\step} \Phi(y, 0, \projM[\Minter](y)) \\
                         &= -\grad \funns(\projM[\Minter](y)),
  \end{align}
  which concludes the proof.
\end{proof}

\subsubsection{\GBedit{Part 2: normal optimality}}

The previous lemma shows that for every $(y,\step)$ one can find a point $\cm(\vxinter, \step)$ \emph{on the manifold} $\Minter$ that solves the tangent part of the optimality condition~\eqref{eq:proxOCtangent}.
The next lemma determines the values of $y$ and $\step$ for which the whole optimality condition
\begin{align}\label{eq:proxOCri}
  y \in \cm(\vxinter, \step) + \step\ri \partial \funns(\cm(\vxinter, \step))
\end{align}
holds, as illustrated in \Cref{fig:proxneighcurve}.

\begin{figure}[t]
  \centering
  \begin{subfigure}[t]{0.45\textwidth}
    \centering
    \includegraphics[width=0.7\textwidth, page=1]{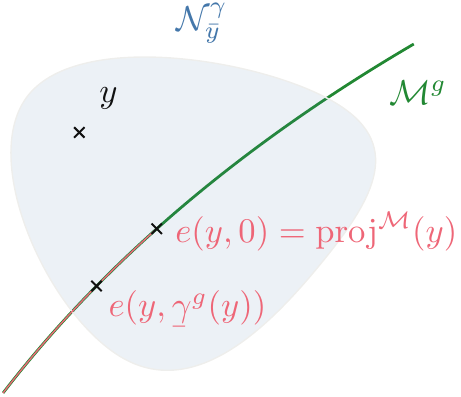}
    \caption{the curve $\step \mapsto \cm(\vxinter, \step)$ on $\M$.\label{fig:proxneighcurve}}
  \end{subfigure}\hfill
  \begin{subfigure}[t]{0.45\textwidth}
    \centering
    \includegraphics[width=0.7\textwidth, page=2]{ipe_proxstable.pdf}
    \caption{the curve $\step\mapsto \prox_{\step \funns}(y)$ on $\M$\\ for $\step\ge\stepLowBnd(\distM[\Minter](y))$.\label{fig:proxneighprox}}
  \end{subfigure}
  \vspace*{-3ex}
  \caption{Illustration of \Cref{lem:fullcurve} and its consequences.\label{fig:proxneigh}}
\end{figure}

\begin{lemma}\label{lem:fullcurve}
  Consider a function $\funns$, a point $\crit[\vxinter]\in\interSpace$ and a manifold $\Minter$ such that $\funns$ is partly smooth at $\crit[\vxinter]$ relative to $\Minter$ and that $\funns$ satisfies \Cref{prop:normalascent} at $\crit[\vxinter]$.
  Let $\cm$ denote a smooth $\M$-valued application defined on a neighborhood of $(\crit[\vxinter], 0)$ provided by~\cref{lem:smoothcurve}.
  Then, there exists $\const>0$ such that:
  \begin{enumerate}
    \item for all $\step\in[0, \const]$, $\cm(\crit[\vxinter], \step)$ verifies~\eqref{eq:proxOCri} with $y = \crit[\vxinter]$,
    \item for all $\step\in[0, \const]$, there exists a neighborhood $\N_{\crit[\vxinter]}^{\step}$ of $\crit[\vxinter]$ such that, for all $y\in\N_{\crit[\vxinter]}^{\step}$, $\cm(\vxinter, \step)$ verifies~\eqref{eq:proxOCri},
  \end{enumerate}
  Further assume that $\funns$ satisfies~\Cref{prop:curve} at $\crit[\vxinter]$ with constant $\constcurve$, then
  \begin{enumerate}\setcounter{enumi}{2}
    \item there exist $\stepUpBnd^{\funns}>0$ and a neighborhood $\N_{\crit[\vxinter]}$ of $\crit[\vxinter]$ %
     such that for all $\vxinter\in\N_{\crit[\vxinter]}$
    \begin{align}
      \cm(\vxinter, \step) \text{ verifies~\eqref{eq:proxOCri} for all } \step \in [\stepLowBnd^{\funns}(\distM[\Minter](\vxinter)), \stepUpBnd^{\funns}],
    \end{align}
          where $\cri \geq 0$ %
          and
      $\stepLowBnd^{\funns}(t) = \frac{\cri}{2\constcurve} \left( 1- \sqrt{1 - \frac{4\constcurve t}{\cri^{2}}} \right) = \frac{1}{\cri}t + \frac{\constcurve}{\cri^{3}}t^{2} + o(t^{2})$.
  \end{enumerate}
\end{lemma}

The proof consists in finding the points $\vxinter, \step$ such that $0\in\ri\Psi(\vxinter, \step)$, where the mapping $\Psi:\interSpace \times \bbR \to\cup_{\vx\in\Minter}\normal{\vx}{\Minter}$ is defined as
\begin{align}
  \Psi(\vxinter, \step) = \proj_{\normal{\cm(\vxinter, \step)}{\Minter}}\left(\frac{1}{\step}\left(\cm(\vxinter, \step) - \vxinter\right) + \partial \funns\left(\cm(\vxinter, \step)\right)\right).
\end{align}
Items \emph{i)} and \emph{ii)} are shown by extending the property $0\in\Psi(\crit[\vxinter], 0)$ to a neighborhood of $(\crit[\vxinter], 0)$, using the inner-semicontinuity properties of $\Psi$.
\GBedit{We refer to~\cite[Def.~5.4]{rockafellar2009variational} for an exposition of the notions of continuity of set-valued mappings.}
We then derive explicit bounds on the interval of steps such that $0\in\ri\Psi(\vxinter, \step)$: for a fixed $\vxinter \in \N_{\crit[\vxinter]}$, when $\step$ decreases past some value, say $\stepLow(y)$, the condition $0\in\ri\Psi(\vxinter, \step)$ no longer holds.
Precisely at $\stepLow(y)$, $0$ lies on the (relative) boundary of $\Psi(\vxinter, \stepLow(y))$: denoting $\rbd S \defeq S \setminus \ri S$ the relative boundary of set $S$,
  \begin{align}
    0 \in \rbd\proj_{\normal{\cm(\vxinter, \stepLow(\vxinter))}{\Minter}}\left(\frac{1}{\stepLow(\vxinter)}\left(\cm(\vxinter, \stepLow(\vxinter)) - \vxinter\right) + \partial \funns\left(\cm(\vxinter, \stepLow(\vxinter))\right)\right).
  \end{align}
  Denoting $\partial^{N} \funns(p) \defeq \proj_{\normal{p}{\Minter}}(\partial \funns(p))$ the projection of the subdifferential on the normal space of its structure manifold and taking norms yields:
  \begin{align}
    \| \proj_{\normal{\cm(y, \stepLow(y))}{\Minter}}(y-\cm(y, \stepLow(y)))\| &\ge \stepLow(y) \inf_{\vn \in \rbd{\partial^{N} \funns(\cm(y, \stepLow(y)))}} \|\vn\| \\
                                                                                                &\ge \stepLow(y) \underbrace{\inf_{p\in\N_{\crit[\vxinter]}}\inf_{\vn \in \rbd{\partial^{N} \funns(p)}} \|\vn\|}_{\defeq \cri}.
  \end{align}
  \GBtwodelete{\GBreplace{Since $0 \in \ri \proj_{\normal{\crit[\vxinter]}{\Minter}} \partial \funns(\crit[\vxinter])$ and $\partial \funns$ is inner-semicontinuous, the former property actually holds on a neighborhood of $\crit[\vxinter]$ in $\Minter$, thus making the constant $\cri$  positive.}
  {By partial smoothness, $\partial \funns$ is continuous on $\Minter$ at $\crit[\vxinter]$, and thus in particular inner-semicontinuous.
    The inclusion $0 \in \ri \proj_{\normal{\crit[\vxinter]}{\Minter}} \partial \funns(\crit[\vxinter])$ therefore holds on a neighborhood of $\crit[\vxinter]$ on $\Minter$~\cite[Lemma 20]{daniilidis2006geometrical}, thus making the constant $\cri$  positive.}}
  \GBtwoedit{
    We note that the constant $\cri$ is positive.
    Indeed, $\vxinter \mapsto \grad \funns(\vxinter) = \proj_{\normal{\vxinter}{\Minter}} \partial \funns(\vxinter)$ is a continuous selection of the affine hull of $\partial \funns$, and $\grad \funns(\crit[\vxinter]) \in \ri \partial \funns(\crit[\vxinter])$ by \cref{prop:normalascent}.
    Lemma 20 from~\cite{daniilidis2006geometrical} then guarantees that $\grad \funns(\vxinter) \in \ri \partial \funns(\vxinter)$ for $\vxinter$ close enough to $\crit[\vxinter]$.
    Projecting back on the normal space at $\vxinter$ provides the inclusion $0 \in \ri \proj_{\normal{\crit[\vxinter]}{\Minter}} \partial \funns(\crit[\vxinter])$ a neighborhood of $\crit[\vxinter]$ on $\Minter$.
    This implies positivity of $\cri$, reducing the size of $\N_{\crit[\vxinter]}$ if necessary.}
  We note that this kind of quantity also appears as the \emph{modulus of identifiability} in the recent~\cite[Def. 2.3]{lewis2022identifiability} where it has the same property: its positivity enables the identification of the associated structure manifold.

  Using \Cref{prop:curve}, the left-hand side is upper
  bounded by a simpler %
  expression:
  \begin{align}
    \constcurve \stepLow(\vxinter)^{2} + \distM[\Minter](\vxinter) \ge \cri \stepLow(\vxinter), \quad \text{ that is } \quad \stepLow(\vxinter) \le
    \frac{\cri}{2\constcurve} \left( 1- \sqrt{1 - \frac{4\constcurve \distM[\Minter](\vy)}{\cri^{2}}} \right),
  \end{align}
  which provides the expression for $\stepLowBnd^{\funns}$ used in the lemma.

\begin{proof}
  \noindent\emph{Item i)}
  We first consider $\Psi_{\crit[\vxinter]}(\cdot) = \Psi(\crit[\vxinter], \cdot)$.
  Since $\crit[\vxinter]\in\Minter$, \cref{lem:smoothcurve} tells us that $\cm(\crit[\vxinter], \step) = \crit[\vxinter] - \step \grad \funns(\crit[\vxinter]) + o(\step)$, and thus
  \begin{align}
    \Psi_{\crit[\vxinter]}(0) = \proj_{\normal{\crit[\vxinter]}{\Minter}}\left(-\grad \funns(\crit[\vxinter]) + \partial \funns(\crit[\vxinter])\right) = \proj_{\normal{\crit[\vxinter]}{\Minter}}(\partial \funns(\crit[\vxinter]))
  \end{align}
  where we used that $\grad \funns(\crit[\vxinter])\in\tangent{\crit[\vxinter]}{\Minter}$ is orthogonal to $\normal{\crit[\vxinter]}{\Minter}$.
  \Cref{prop:normalascent} provides that  $0 \in \ri \Psi_{\crit[\vxinter]}(0)$.
  We now turn to show that there exists $\constalt$ such that, for all $\step \in [0, \constalt]$, $0 \in \ri \Psi_{\crit[\vxinter]}(\step)$.

  By contradiction, assume there exist a sequence $\curr[\step]\to 0$ such that $0\notin \ri\Psi_{\crit[\vxinter]}(\curr[\step])$.
  This means that there exists a sequence of unit norm vectors $( \curr[s] )$ such that for all $\ite$,
  \begin{align}\label{eq:sepbypsi}
     \langle \curr[s], z \rangle \le 0 \text{ for all } z \in \Psi_{\crit[\vxinter]}(\curr[\step]).
  \end{align}
  As a bounded sequence, $\curr[s]$ admits at least one limit point, say $\crit[s]$.
  Take $\crit[z]\in\Psi_{\crit[\vxinter]}(0)$.
  The continuity of $\partial \funns$ (by partial smoothness, item iv), of $\step\mapsto (\cm(\crit[\vxinter], \step) - \crit[\vxinter])/\step$ (by smoothness of $\cm$), and of $\step \mapsto \proj_{\normal{\cm(\crit[\vxinter], \step)}{\Minter}}$  (by smoothness of $\Minter$) yield the continuity of $\Psi_{\crit[\vxinter]}$ as a set-valued map.
  This mapping is thus inner-semicontinuous~\cite[Def.~5.4]{rockafellar2009variational},  so there exists a sequence $\curr[z] \in \Psi_{\crit[\vxinter]}(\curr[\step])$ such that $\curr[z]$ converges to $\crit[z]$.
  Taking the correct subsequence and renaming iterates, we can write $\curr[s] \to \crit[s]$ and $\curr[z] \to \crit[z]$.
  Equation~\eqref{eq:sepbypsi} provides $\langle \curr[s], \curr[z] \rangle\le 0$ for all $\ite$, which gives at the limit $\langle \crit[s], \crit[z] \rangle \le 0$.
  This actually holds for all $\crit[z] \in \Psi_{\crit[\vxinter]}(0)$: $\crit[s]$ separates $0$ and $\Psi(0)$, which contradicts %
  $0\in\ri\Psi_{\crit[\vxinter]}(0)$.

  Finally, let us take the constant $\const$ such that $[0,\const]$ is included in both $[0,\constalt]$ and the neighborhood of $0$ provided by~\cref{lem:smoothcurve}.
  Then, for any $\step \in [0,\const]$, adding the two orthogonal inclusions $0 \in \ri \Psi_{\crit[\vxinter]}(\step)$ and $0 = \Phi(\vxinter,\step,\GBreplace{c}{\cm}(\vxinter, \step))$,  we obtain that $\cm(\crit[\vxinter], \step)$ verifies~\eqref{eq:proxOCri} with $y = \crit[\vxinter]$.

  \medskip

  \noindent\emph{Item ii)}
  Let $\step\in[0, \const]$.
  We turn to show the existence of a neighborhood $\N_{\crit[\vxinter]}^{\step}$ of $\crit[\vxinter]$ such that, for all $\vxinter \in \N_{\crit[\vxinter]}^{\step}$, $\cm(\vxinter, \step)$ verifies~\eqref{eq:proxOCri}.
  By contradiction, assume that there exists a sequence $(\curr[\vxinter])$ that converges to $\crit[\vxinter]$ such that~\eqref{eq:proxOCri} fails for $(\curr[\vxinter], \step)$.
  Since the tangent component of~\eqref{eq:proxOCri} does hold, necessarily $0 \notin \ri \Psi(\curr[\vxinter], \step)$.
  However, the mapping $\vxinter \mapsto \Psi(\vxinter, \step)$ is inner-semicontinuous (from the same arguments as in the proof of \emph{item i)} and there holds $0\in\ri\Psi(\crit[\vxinter], \step)$.
  A reasoning similar to that of \emph{item i)} reveals the contradiction.

  \medskip

  \noindent\emph{Item iii)}
  Define $\N_{\crit[\vxinter]}$ a neighborhood of $\crit[\vxinter]$ and $\stepUpBnd^{\funns} $ a positive constant such that \Cref{prop:curve} applies over $\N_{\crit[\vxinter]}\times [0, \stepUpBnd^{\funns}]$,
  \GBreplace{$\N_{\crit[\vxinter]}$ is contained in $\cup_{\step \in [0, \const]}\N_{\crit[\vxinter]}^{\step} \cap \N_{\crit[\vxinter]}^{\const}$ }{}and
  $0\in \ri \Psi(\vxinter, \step)$ holds for all $(\vxinter, \step) \in \N_{\crit[\vxinter]}\times [0, \stepUpBnd^{\funns}]$.
  \GBreplace{This last}{The second} condition can be met on a nontrivial neighborhood of $(\crit[\vxinter], 0)$: it holds at that point, and $\Psi$ is inner-semicontinuous ($\cm(\vxinter, \step)$ lies on $\Minter$ and $\partial \funns$ is inner-semicontinuous by partial smoothness of $\funns$).

  Let $\vxinter\in\N_{\crit[\vxinter]}$ and \GBreplace{$\step>0$ such that $\stepLowBnd^{\funns}(\distM[\Minter](y)) \le \step \le \stepUpBnd^{\funns}$}{$\step \in [\stepLowBnd^{\funns}(\distM[\Minter](y)), \stepUpBnd^{\funns}]$}.
  \GBedit{
  We show that $0 \in \ri \Psi(y, \gamma)$, that is
  \begin{equation}
    \proj_{\normal{\cm(\vxinter, \step)}{\Minter}}(\vxinter - \cm(\vxinter, \step)) \in \step \ri \partial^{N} \funns\left(\cm(\vxinter, \step)\right).
  \end{equation}
  Combining this with the orthogonal inclusion $0 = \Phi(y, \gamma, \cm(y,
  \gamma))$ yields the claim.
  }

  The \GBreplace{lower bound on $\step$}{inequality $\stepLowBnd^{\funns}(\distM[\Minter](y)) \le \step$} implies \GBreplace{that }{}$\constcurve \step^{2} + \distM(\vxinter) \le \step \cri$.
  We have successively \GBreplace{by \Cref{prop:curve}}{by definition of $\N_{\crit[\vxinter]}$} and the above bound that
  \begin{align}
    \| \proj_{\normal{\cm(\vxinter, \step)}{\Minter}}(\vxinter-\cm(\vxinter, \step))\| \le \distM(\vxinter) + \constcurve \step^{2} & \le \step \cri \\
    &\le \step \inf \{ \|n\|, n \in \rbd\partial^{N} \funns(e(\vxinter, \step))\}.
  \end{align}
  This means that $\proj_{\normal{\cm(\vxinter, \step)}{\Minter}}(\vxinter-\cm(\vxinter, \step))$ belongs to the ball of center $0$ and radius $ \step \inf \{ \|n\|, n \in \rbd\partial^{N}(\funns(e(\vxinter, \step)))\}$ in $\normal{\cm(\vxinter, \step)}{\Minter}$.
  \GBreplace{In addition}{Besides}, this ball is included in $\step \partial^{N}(\funns(e(\vxinter, \step))$ since $0\in\partial^{N} \funns(e(\vxinter, \step)$ by definition of $ \N_{\crit[\vxinter]}$.
  Therefore, $0\in \ri\Psi(\vxinter, \step)$ for all $\vxinter \in \N_{\crit[\vxinter]}$ and $\step \in [\stepLowBnd^{\funns}(\distM[\Minter](y)), \stepUpBnd^{\funns}]$.
\end{proof}

\subsubsection{\GBedit{Part 3: From the intermediate space to the input space}}

\FIedit{
To conclude the proof of \Cref{th:proxintspace}, we will first identify the curve  $\cm(\vxinter, \step)$ to $\prox_{\step \funns}(y)$ and thus prove that it belongs to the sought manifold, as illustrated in \cref{fig:proxneighprox}. Then, this intermediate identification result is brought back to the input space using transversality. 
}

\begin{proof}
  The standing assumptions allow to call~\Cref{lem:fullcurve} at point $\mapping(\crit[\vx])$ with manifold $\Minter$.
  This yields the neighborhood $\N_{\mapping(\crit[\vx])}$,  constants $\stepUpBnd^{\funns}$ and $\const$, a function $\stepLowBnd^{\funns}$, and a smooth mapping $\cm : \N_{\mapping(\crit[\vx])}\times [0, \const] \to \Minter$ such that, for $\vxinter\in\N_{\mapping(\crit[\vx])}$ and $\step \in [\stepLowBnd^{\funns}(\distM[\Minter](\vxinter)), \stepUpBnd^{\funns}]$, $\cm(\vxinter, \step)$ verifies the optimality condition~\eqref{eq:proxOCri} of $\cm(\vxinter, \step) = \prox_{\step \funns}(y)$.
  Besides, since $\funns$ is prox-regular and prox-bounded at point $\mapping(\crit[\vx])$, these properties also hold on a neighborhood of that point.
  Under these conditions,~\Cref{prop:wellbehavedprox} allows to recover the equality $\cm(\vxinter, \step) = \prox_{\step \funns}(\vxinter)$.
  Take $\N_{\crit[\vx]} = \mapping^{-1}(\N_{\mapping(\crit[\vx])})$, a neighborhood of $\crit[\vx]$ as the preimage of a neighborhood of $\mapping(\crit[\vx])$ by the continuous $\mapping$.
  For all $\vx \in \N_{\crit[\vx]}$,
  \begin{align}
    \prox_{\step \funns}(\mapping(\vx)) \in \Minter \text{ for all } \step\in [\stepLowBnd^\funns(\distM[\Minter](\mapping(\vx))), \stepUpBnd^\funns].
  \end{align}

  We turn to show that, for some constant $\cmap >0$, there holds $\distM[\Minter](\mapping(\vx)) \le \cmap \distM(\vx)$ for all $\vx\in\N_{\crit[\vx]}$.
  Let $\vx\in\N_{\crit[\vx]}$ and $\vxman = \proj_{\M}(x)$, so that $\distM(\vx) = \|\vxman - \vx\|$.
  Using successively that $\mapping(\vx^{\M}) \in \Minter$ and smoothness of $\mapping$, there holds for $\vx$ near $\crit[\vx]$
  \begin{align}
    \distM[\Minter](\mapping(\vx)) &\le \|\mapping(\vx) - \mapping(\vxman)\| \\
                                   &\le \|\Jac_{\mapping}(\vxman)\cdot (\vx - \vxman)\| + \bigoh(\|\vx - \vxman\|^{2}) \\
                                   &\le \left(\sup_{\vn \in \normal{\vxman}{\M}, \|\vn\| = 1} \|\Jac_{\mapping}(\vxman) \cdot \vn\| \right) \|\vx - \vxman\| + \bigoh(\|\vx - \vxman\|^{2}) \\
                                   &\le \underbrace{\left(\sup_{\vxalt \in \N_{\crit[\vx]}}\sup_{\vn \in \normal{\vxalt}{\M}, \|\vn\| = 1} \|\Jac_{\mapping}(\vxalt) \cdot \vn\| \right)}_{\conststep} \|\vx - \vxman\| + \bigoh(\|\vx - \vxman\|^{2}).
  \end{align}
  \GBedit{
    We show by contradiction that the constant $\conststep$ is positive.
    If $\conststep = 0$, there exists $\vn \in \normal{\crit}{\M}$ of unit norm such that $\D \mapping(\crit) \vn = 0$.
    By \cref{eq:normalspaceintertoinput}, we have $\vn = \D \mapping(\crit)^{*} \vnalt$ for some $\vnalt\in\normal{\mapping(\crit)}{\Minter}$, so that $\D \mapping(\crit) \D \mapping(\crit)^{*} \GBtwoedit{\vnalt} = 0$.
    Pre-multiplying by $\vnalt^{*}$ yields $\|\D \mapping(\crit)^{*} \vnalt\|^{2} = 0$: there holds $\vnalt \in \ker({\D \mapping(\crit)}^{{*}}) \cap \normal{\mapping(\crit)}{\Minter}$.
    The transversality condition \cref{eq:transversality} implies $\vnalt = 0$, and in turn $\vn = 0$, which contradicts the fact that this vector has unit length.
  }

  Therefore, for all $\vx \in \N_{\crit[\vx]}$ and a constant $\cmap > \conststep$, there holds $\distM[\Minter](\mapping(\vx)) \le \cmap \distM(\vx)$.
  Monotony of $\stepLowBnd^{\funns}$ implies that $\stepLowBnd^{\funns}(\distM[\Minter](\mapping(\vx))) \le \stepLowBnd^{\funns}(\cmap \distM(\vx))$\GBreplace{.
  Hence}{, which yields} the claimed bounds with
  \begin{align}
    \stepLowBnd(t) = \frac{\cri}{2\constcurve} \left( 1- \sqrt{1 - \frac{4\constcurve\cmap t}{\cri^{2}}} \right)  \quad \text{ and } \quad \stepUpBnd = \stepUpBnd^{\funns}.
  \end{align}

  Finally, we show the existence of positive constants $\epsilon$, $L$ such that
  \begin{align}
    \|\vx - \crit\| \le \epsilon \text{ and } L\|\vx - \crit\| \le \step \le \stepUpBnd \Longrightarrow \prox_{\step \funns}(\mapping(\vx)) \in \Minter.
  \end{align}
  Since $\crit\in\M$, $\distM(\cdot) \le \|\cdot - \crit\|$.
  By monotony and smoothness of $\stepLowBnd$, there exists $L>0$ such that $\stepLowBnd(\distM[\Mopt](\cdot)) \le L\|\cdot - \opt[\vx]\|$ over $\B(\opt, \epsilon)$.
  Reducing $\epsilon$ if necessary so that $L\epsilon < \stepUpBnd$ yields the result.
\end{proof}

\section{\GBreplace{Proposed method}{A local Newton algorithm for nonsmooth composite minimization}}%
\label{sec:method}

In this section, we use the results of~\Cref{sec:identif-prox} to propose an optimization method that locally identifies the structure of a minimizer and converges quadratically to this point.

Recall the basic idea: if the optimal manifold $\opt[\M]$ corresponding to a minimizer $\opt$ is known, the \emph{nonsmooth} optimization problem turns into a \emph{smooth constrained} optimization problem.
In turn, this problem can be solved using algorithms from smooth constrained optimization such as Sequential Quadratic Programming. %

Using this idea and the structure identification mechanism developed in the previous section, we propose a method which: i) uses the proximity operator of $\funns$ to gather structure in the intermediate space, ii) brings back this structure to the input space, and iii) optimizes smoothly along the identified manifold. The resulting algorithm is precisely described in \Cref{sec:desc} and then analyzed in \Cref{sec:conv}.

\subsection{Description of the algorithm}\label{sec:desc}
We proceed to describe the three steps exposed above.
The full algorithm is depicted in~\cref{alg:localcomposition}.

\emph{Gathering structure.}
We showed in \cref{th:proxintspace} that near a qualified point in $\inputSpace$, the operator $\prox_{\step \funns}(\mapping(\cdot))$ provides the optimal structure $\opt[\Minter]$ (in the intermediate space $\interSpace$) for an explicit range of steps.
We thus define from the current iterate $\curr[\vx]\in\inputSpace$ and stepsize $\curr[\step]$ the working manifold $\currinter[\M]$ (in the intermediate space) as the structure of $\prox_{\curr[\step] \funns}(\mapping(\curr[\vx]))$.
One technical point is to guarantee that, after some time, $\curr[\step] \in  [L\|\curr - \opt\|, \stepUpBnd]$ so that the optimal manifold is identified; this is done by decreasing $\curr[\step]$ linearly at each iteration.

\emph{From the intermediate to the input space.}
We now have a structure manifold $\currinter[\M]$ in the intermediate space, and can define $\currFsext[\funns]$, a smooth extension of $\funns$ on $\currinter[\M]$ to $\interSpace$.
Using a local equation\;$\currinter[\maneq]$ of $\currinter[\M]$, we define the smooth map $\curr[\maneq] = \currinter[\maneq] \circ \mapping:\inputSpace\to\bbR^{\curr[p]}$, which locally defines $\curr[\M] = \mapping^{-1}(\currinter[\M])$.
Similarly, a smooth extension of $\funcomp$ on $\curr[\M]$ %
is defined by $\curr[\Fsext] = \currFsext[\funns] \circ \mapping$.

\emph{Optimizing in the input space.}
We can now take steps to minimize the smooth extension $\curr[\Fsext]$ on the smooth set $\curr[\M]$ characterized by $\curr[\maneq](\vx) = 0$:
\begin{align}\label{eq:pbcstrsmooth}
  \min_{\vx\in\inputSpace} \curr[\Fsext](\vx) \quad \text{s.t.} \quad \curr[\maneq](\vx) = 0.
\end{align}
We turn to an elementary version of the traditional second-order Sequential Quadratic Programming methodology; see \eg \cite[Chap.\;14]{bonnans2006numerical}.
At iteration $\ite$, the SQP direction $\currdSQP$ at point $\curr[\vx]$ is defined as the solution of the following quadratic problem:
\begin{align}\label{eq:SQPstep}
  \begin{aligned}
    \currdSQP = \argmin_{d\in\inputSpace} \quad & \langle \nabla\curr[\Fsext](\curr[\vx]), d \rangle + \frac{1}{2}\langle \nabla^{2}_{xx}\curr[L](\curr[\vx], \curr[\Lagmult](\curr[\vx]))d , d \rangle \\
    \textrm{s.t.} \quad & \curr[\maneq](\curr) + \D \curr[\maneq](\curr) d = 0
  \end{aligned}
\end{align}
where $\nabla^{2}_{xx} \curr[L]$ denotes the Hessian of the Lagrangian $\curr[L](\vx, \Lagmult) = \curr[\Fsext](\vx) + \langle \Lagmult, \curr[\maneq](\vx) \rangle$, and  the multiplier $\curr[\Lagmult](\curr)$ defined from the following least-squares problem:
\begin{align}\label{eq:leastsquaremult}
  \curr[\Lagmult](\curr) = \argmin_{\Lagmult\in\bbR^{\curr[p]}} \left\| \nabla \curr[\Fsext](\curr) + \sum_{i=1}^{\curr[p]} \Lagmult_{i} \nabla \maneq_{k,i}(\curr) \right\|^{2}.
\end{align}

Finally, we check that $\curr + \currdSQP$ provides a functional decrease in order to avoid degrading the iterate when the current structure is suboptimal. If the test is not verified, $\curr$ is not updated and $\curr[\step]$ is decreased until a satisfying structure is detected.

\begin{algorithm}
    \caption{General structure exploiting algorithm\label{alg:localcomposition}}
    \begin{algorithmic}[1]
        \Require Pick $\init$ near a minimizer, $\init[\step]$ large enough.
        \Repeat
        \State $\curr[\step] = \frac{\prev[\step]}{2}$
        \State Compute $\prox_{\curr[\step] \funns}(\mapping(\curr))$ and obtain $\currinter[\M]$ locally defined by $\currinter[\maneq]$
        \State $\curr[\maneq] = \currinter[\maneq] \circ \mapping$ (local equation of  $\curr[\M]$), $\curr[\Fsext] = \currFsext[\funns] \circ \mapping$ (smooth extension)
        \State Compute $\currdSQP$ by solving \eqref{eq:SQPstep}
        \If{$\funcomp(\curr + \currdSQP) \le \funcomp(\curr)$}
        \State $\next = \curr + \currdSQP$
        \Else
        \State $\next = \curr $
        \EndIf        
        \Until{stopping criterion}
    \end{algorithmic}
\end{algorithm}

\GBedit{
  \begin{remark}[Complexity of one iteration]
    The main computational cost of one iteration of \cref{alg:localcomposition} consists in the resolution of the quadratic program~\eqref{eq:SQPstep}.
    Its plain resolution incurs a $\bigoh(n^{3})$ complexity.
    However, efficient approaches \emph{reduce} this problem to a quadratic program on the subspace $\ker \D h_{k}(x_{k})$, which has dimension $\dim(\M_{k})$.
    We refer to~\cite[Chap. 14]{bonnans2006numerical} for an in-depth exposition of these techniques.
    The cost of an iteration is thus $\bigoh(\dim(\M_{k})^{3})$.
    In situations where minimizers are highly structured (\ie $\dim(\M^{\star}) \ll n$) this complexity may be comparable with the $\bigoh(n^{2})$ iteration complexity of classical nonsmooth optimization algorithms, such as nonsmooth BFGS~\cite{lewisNonsmoothOptimizationQuasiNewton2013}.
  \end{remark}
}

\subsection{Convergence of~\Cref{alg:localcomposition}}\label{sec:conv}
We proceed to give the result guaranteeing identification and local quadratic  convergence of~\Cref{alg:localcomposition}.

In order to benefit from the quadratic rate of SQP, the elements of~\eqref{eq:SQPstep} should have the minimal regularity typically required by smooth constrained Newton methods (see \eg~\cite[Th.\;14.5]{bonnans2006numerical}); we thus make the following assumption.
\begin{assumption}[Regularity of functions]\label{asm:funreg}
  The smooth extension and the manifold defining map are $\C^{2}$ with Lipschitz second derivatives, and the Jacobian of the constraints is full rank near the solution.
\end{assumption}

In order to focus on the algorithmic originality of the method, we slightly simplify the situation and make the two following algorithmic assumptions.
\begin{assumption}[Nonconvex stability]\label{asm:noncvxstab}
  The iterates of~\Cref{alg:localcomposition} remain in the \GBtworeplace{connex}{connected} component of the sublevel set $\{\vx : \funcomp(\vx) \le \funcomp(\init)\}$ that contains $\opt$.
\end{assumption}
This assumption ensures that an update that decreases the functional value remains in the neighborhood of the minimizer $\opt$.
It is naturally satisfied when $\funcomp$ is convex, or when $\opt$ is a global minimizer of $\funcomp$ and $\init$ is close enough to $\opt$.

\begin{assumption}[No Maratos effect]\label{asm:nomaratos}
  The iterates of~\Cref{alg:localcomposition} are such that a step $d$ that makes $\vx + d$ quadratically closer to $\vx$ yields descent: $\funcomp(\vx + d) \le \funcomp(\vx)$.
\end{assumption}
In smooth constrained optimization, getting closer (even at quadratic rate) to a minimizer does not imply decrease of objective value and constraint violation (measured by a merit function).
This so-called Maratos effect (see \eg~\cite{bonnans2006numerical}) is one of the main difficulties in globalizing SQP schemes, which is out of the scope of the current paper. We thus assume this effect does not affect our algorithm in theory, and use in practice one of the successful refinements, as discussed in~\Cref{sec:numsetup}.

We are now ready for the main convergence result of~\Cref{alg:localcomposition}, which establish that, after some finite time, the iterates identify exactly the optimal manifold and converge to the minimizer at a quadratic rate.

\begin{theorem}[Exact identification and quadratic convergence]\label{th:algIdentifQuadrate}
  Consider a function $\funcomp = \funns\circ \mapping$ and $\opt$ a strong minimizer,\footnote{There exists $\eta>0$, $\epsilon>0$ such that  $\funcomp(\vx) \ge \funcomp(\opt) + \eta \|\vx - \opt\|^{2}$ for all $\vx \in\B(\opt, \epsilon)$.} qualified relative to the optimal manifold $\opt[\M]$.
  Assume that the smooth extension $\Fsext$ of $\funcomp$ relative to $\opt[\M]$ and the corresponding manifold defining map $\maneq$ satisfy~\Cref{asm:funreg}.

  If $\init$ and $\funcomp(\init)$ are close enough to $\opt$ and $\funcomp(\opt)$, $\init[\step]$ is large enough and the simplifying algorithmic~\Cref{asm:noncvxstab,asm:nomaratos} hold, then there exists $C>0$ such that the iterates $(\curr, \curr[\M])$ generated by~\Cref{alg:localcomposition} verify:
  \begin{align}
    \curr[\M] = \opt[\M] \quad \text{and} \quad \|\next - \opt\| \le C \| \curr - \opt \|^{2} \quad \text{for all } k \text{ large enough. }
  \end{align}
\end{theorem}

The proof of this result consists in two steps.
We first show the existence of a neighborhood of initialization on which the proximity operator will eventually identify the optimal manifold, once the stepsize has been sufficiently decreased.
From this point onward, we prove that the SQP step provides a quadratic improvement and that the stepsize policy makes the manifold identification stable.

\begin{proof}
  \noindent\emph{Local identification of the optimal structure.}
  By \cref{th:proxintspace}, there exists a ball centered around $\opt$ of radius $\epsilon_1 > 0$ and two positive constants $L$, $\Gamma$ such that, for all $\vx \in \ball(\opt, \epsilon_1)$ and $\step\in [L\|\vx-\opt\|, \stepUpBnd]$, $\prox_{\step \funns}(\mapping(\vx))$ belongs to the optimal manifold $\opt[\Minter] = \mapping(\Mopt)$.

  \medskip

  \noindent\emph{Local quadratic convergence of SQP on the optimal structure.}
  Let us assume that the optimal manifold has been identified. The least square multiplier $\Lagmult$ is defined by the optimality condition of~\eqref{eq:leastsquaremult}:
    \begin{align}
      \Lagmult(\vx) = - [\Jac_{h}(\vx) \Jac_{h}(\vx)^{\top}]^{-1} \Jac_{h}(\vx) \nabla \Fsext (\vx).
    \end{align}
    and since $\maneq$ is smooth and its Jacobian is full-rank near $\opt$, $\Lagmult$ is a Lipschitz continuous function near $\opt$.

    Since $\opt$ is a strong minimizer of $\funcomp$, the Hessian of the Lagrangian restricted to the tangent space is positive definite.
      Indeed, since $\opt$ is a strong minimizer of $\funcomp$ on $\opt[\M]$, the Riemannian Hessian relative to the optimal manifold is positive definite.
      With the choice of multiplier~\eqref{eq:leastsquaremult}, the Riemannian Hessian is exactly the Hessian of the Lagrangian restricted to the tangent space to $\opt[\M]$ at $\opt$ (see~\cite[Sec. 7.7]{boumal2022intromanifolds}), which is thus itself positive definite.

  Thus, using the local quadratic convergence of SQP~\cite[Th. 14.5]{bonnans2006numerical}, we get that there exists a ball centered around $\opt$ of radius $\epsilon_2 > 0$ such that the SQP step computed at a point $\vx$ in that neighborhood relative to the optimal manifold provides a quadratic improvement towards $\opt$. Reducing $\epsilon_2$ if necessary, we can in addition have that the convergence is at least linear with rate $1/2$.

  \medskip

  \noindent\emph{Initialization, identification, and quadratic convergence.}
  Let $\epsilon = \min(\epsilon_1,\epsilon_2,\Gamma/(2L))$.
  We will now show that initializing with $\init \in \{\vx : \funcomp(\vx) \le \funcomp(\opt) + \eta\epsilon^{2}\}$ and $\init[\step] \ge \Gamma$ provides the claimed behavior.

  First, the functional decrease test of the algorithm and~\Cref{asm:nomaratos} guarantee that all iterates satisfy $\funcomp(\curr) \le \funcomp(\init)$.
  Using that $\opt$ is a strong minimizer, we get that $\eta \|\curr - \opt\|^{2} \le \funcomp(\curr) - \funcomp(\opt) \le \funcomp(\init) - \funcomp(\opt)\le \eta \epsilon^{2}$, and thus that the iterates remain in $\B(\opt, \epsilon)$.

  Second, as $L\|\vx - \opt\| \le \Gamma / 2$ for all $\vx\in\B(\opt, \epsilon)$ by construction, the fact that $\init[\step] > \Gamma$ and $(\curr[\step])$ decreases with geometric rate $1/2$ implies that there exists $K$ such that $L\|\vx_{K} - \opt\| \le \step_{K} \le \Gamma$.

  Now, assume that at iteration $k\ge K$, $L\|\vx_{k} - \opt\| \le \step_{k} \le \Gamma$. Since $\vx_k\in\ball(\opt, \epsilon_1)$, we have from above that $\opt[\M]$ is identified.
  Thus, the SQP step is performed relative to the optimal manifold and $\curr+\currdSQP$ brings a linear improvement of factor $1/2$ at least.
  \Cref{asm:noncvxstab} ensures that $\funcomp(\curr + \currdSQP) \leq \funcomp(\curr)$ so that $\next = \curr + \currdSQP$ and thus
  \begin{align}
    L \| \next - \opt \| \le \frac{L}{2} \| \curr - \opt \| \le \frac{\curr[\step]}{2} = \next[\step].
  \end{align}
  This shows that  $L\|\next[\vx] - \opt\| \le \next[\step] \le \Gamma$, which completes the induction.
  We get that $\curr[\step]\in [L\|\curr-\opt\|, \stepUpBnd]$ for all $\ite\ge K$.
  Finally, we have that for all $k\ge K$, $\curr[\M] = \opt[\M]$ and $\next$ is quadratically closer to $\opt$ than $\curr$.
\end{proof}

 \emph{Direct generalizations.}
  \Cref{th:algIdentifQuadrate} actually holds for any decrease factor of $\curr[\step]$ in $(0, 1)$ with the presented SQP update, or actually any superlinearly convergent update (\eg a quasi-Newton type update).
  The above result is also readily adapted to an update that converges merely linearly, as long as its rate of convergence is faster than that of $\curr[\step]$.
  This opens the possibility of using SQP methods that rely only on first-order information (see \eg~\cite{bolteMajorizationMinimizationProceduresConvergence2016}).

\section{Numerical experiments}%
\label{sec:numexps}

In this section, we provide numerical illustrations for our results.
Our goal here is twofold:
\begin{enumerate}
  \item to illustrate the identification of the optimal manifold by the proximity operator near a minimizer as provided by~\Cref{th:proxintspace};
  \item to demonstrate the applicability of~\Cref{alg:localcomposition} and observe the quadratic rates predicted by~\Cref{th:algIdentifQuadrate} on our running examples.
\end{enumerate}

\subsection{Test problems}%
\label{sec:numexps_pbs}

We first consider the minimization of a pointwise maximum of smooth functions \eqref{example:maxofsmooth}:
\begin{align}
  \min_{\vx\in\inputSpace} \max_{i = 1, \ldots, m} (\mapping_{i}(\vx)).
\end{align}
We take the celebrated \texttt{MaxQuad} instance, where $n = 10$, $m = 5$ and each $\mapping_{i}$ is quadratic convex, making the whole function $\funcomp$ convex~\cite[p. 153]{bonnans2006numerical}.
In this instance, the optimal manifold is $\MI$ with $I = \{2, 3, 4, 5\}$. %

Second, we consider the minimization of the maximum eigenvalue of an affine mapping \eqref{example:eigmax}:
\begin{align}
  \min_{\vx\in\inputSpace} ~\lammax\left(A_{0} + \sum_{i=1}^{n} x_{i}A_{i}\right).
\end{align}
We take $n=25$ and we generate randomly $n+1$ symmetric matrices of size $50$.
In this  %
instance, the multiplicity of the maximum eigenvalue at the minimizer is $r = 3$.

\subsection{Numerical setup}%
\label{sec:numsetup}

All the algorithms are implemented in Julia~\cite{bezanson2017julia}; experiments may be reproduced using the code available online\footnote{
  \GBedit{See \url{https://github.com/GillesBareilles/LocalCompositeNewton.jl} for \cref{alg:localcomposition} and \url{https://github.com/GillesBareilles/NonSmoothSolvers.jl} for the baselines.}
}.

\emph{Algorithm.} %
For the initialization of \Cref{alg:localcomposition}, we set $\init[\step]$ as the smallest $\step$ such that $\prox_{\step \funns}(\mapping(\init))$ has the most structure (\eg if $\funns=\max$, we increase $\step$ until the output of the proximity operator sets all coordinates to the same value\GBedit{, and if $\funns=\lammax$, we increase $\step$ until the multiplicity of the maximal eigenvalue of the output of the proximity operator is maximal}).
We solve the quadratic subproblem~\eqref{eq:SQPstep} providing the SQP step by the reduced system approach presented in~\cite[p.\;133]{bonnans2006numerical}.
Tangent vectors are expressed in an orthonormal basis of the nullspace of the Jacobian of the constraints at the current iterate.
At iterate $\curr$, a second-order correction step $\dcorr[\curr]$ is added to the SQP step $\dSQP[\curr]$.
It is obtained as $\dcorr[\curr] = \argmin_{d\in\inputSpace}\{\|\maneq(\curr) + \Jac_{\maneq}(\curr)\;d\|, \text{ s.t. } d\in\Im\; \Jac_{\maneq}(\curr)^{\top}\}$.
The full-step is thus $\curr + \dSQP[\curr] + \dcorr[\curr]$.

\paragraph{Baselines}
For the two nonsmooth problems, we compare with the nonsmooth BFGS algorithm of \cite{lewisNonsmoothOptimizationQuasiNewton2013} (nsBFGS) and the gradient sampling algorithm~\cite{burke2020gradient}. %
The nsBFGS method is not covered by any theoretical guarantees; it is known to perform relatively well in practice, often displaying a linear rate of convergence.
In contrast, the Gradient Sampling algorithm generates with probability one a sequence of iterates for which all cluster points are Clarke stationary for $\funcomp$~\cite[Th. 3.1]{burke2020gradient}.\footnote{This holds when $\funcomp$ is locally Lipschitz over $\inputSpace$ and lower bounded, the algorithm iterates indefinitely and the sampling radius decreases to $0$.}

\GBedit{
Other methods could be considered as relevant baselines.
In particular, the minimization of convex composite functions can be tackled with dedicated bundle methods~\cite{sagastizabalCompositeProximalBundle2013}.
Alternatively, some approaches try to estimate and use the optimal structure $\opt[\M]$, leading to potential superlinear convergence: \cite{womersleyAlgorithmCompositeNonsmooth1986} for the maximum of smooth functions, \cite{nollSpectralBundleMethods2005,helmbergSpectralBundleMethod2014} for the maximum eigenvalue, and \cite{mifflin2005algorithm} for general convex functions.
However, the superlinear speed of these methods hinges on the correct identification of the optimal manifold $\opt[\M]$, which is done only heuristically.
We do not include these methods in our numerical comparison since they are rather advanced, and \GBtwoedit{thus} difficult to implement and tune efficiently. %
}

\paragraph{Oracles} %
Traditional methods for nonsmooth optimization, and notably bundle methods, require a first-order oracle:
\begin{align}
  \vx \mapsto (\funcomp(\vx), v) \quad \text{ where } v \in \partial \funcomp(\vx)
\end{align}
while Gradient Sampling and nsBFGS require additionally to know if $\funcomp$ is differentiable at point $\vx$.
\Cref{alg:localcomposition} requires rather different information oracles:
\begin{align}
  \vx &\mapsto \funcomp(\vx) \\
  \vx &\mapsto \Minter \ni \prox_{\step \funns}(\mapping(\vx)) \\
  \M, \vx & \mapsto \maneq(\vx), \Jac_{\maneq}(\vx), \nabla \Fsext(\vx), \nabla^{2} L(\vx, \lambda).
\end{align}
The second part of the oracle provides the candidate structure at point $\vx$.
The last part of the oracle, which requires a point \emph{and a candidate structure}, provides the second-order information of $\funcomp$ required by the SQP step.

\subsection{Experiments}%
\label{subsec:observations}

\Cref{fig:suboptvstime} reports the suboptimality of the considered methods in terms of CPU time and each marker corresponds to one iteration.
All algorithms are initialized at a point $\init$ obtained by running nsBFGS for several iterations.

Our algorithm compares favorably to nsBFGS and Gradient Sampling: it converges in a handful of iterations and less time. %
Note that this happens even though the iteration cost of our algorithm is higher than that of the other methods.
Indeed, the oracles of our method are more complex and a quadratic problem needs to be solved, while the iteration cost of nsBFGS and Gradient Sampling is dominated by the computation of function values and subgradients at each trials of the linesearch.

In terms of identification, our method finds the correct manifold at the first iteration for MaxQuad, and at the third iteration for Eigmax. %
From that point, the iterates of~\Cref{alg:localcomposition} reach machine precision in 3 iterations.
This illustrates the quadratic convergence, and 
supports the idea that, for nondifferentiable problems as well, it is worth computing higher-order information to get fast local methods.

\Cref{fig:stepitBigFloat} allows to observe the identification of the algorithm and the quality of the bounds of~\Cref{th:proxintspace}.
For each iterate $\curr$ of \Cref{alg:localcomposition}, we report the current step $\curr[\step]$ along with the minimal and maximal steps $\stepLow(\curr), \stepUp(\curr)$ such that $\prox_{\step \funns}(\mapping(\curr))$ belongs to the optimal manifold.\footnote{To better illustrate the local behavior of our method, we also ran the algorithms with a high precision floating type. Details and corresponding experiments can be found in \cref{apx:highprec}.}
A first remark is that, as predicted by~\Cref{th:algIdentifQuadrate}, the pair $\curr, \curr[\step]$ satisfies the identification condition $\curr[\step]\in [L\|\curr - \opt, \stepUpBnd]$ after a few iterations.
We also observe that $\stepUp(\curr)$ is near constant and that $\stepLow(\curr)$ converges to zero linearly with $\|\curr - \opt\|$, as predicted by our result.
Finally, we note that even though the initial point is not structured and away from the minimizer ($\|\init - \opt\|\approx 10^{-2}$), the %
initialization of $\init[\step]$ ensures a quick identification.

\begin{figure}[t]
  \centering
  \begin{subfigure}[t]{0.49\textwidth}
    \centering
    \includetikzexpenumcaption{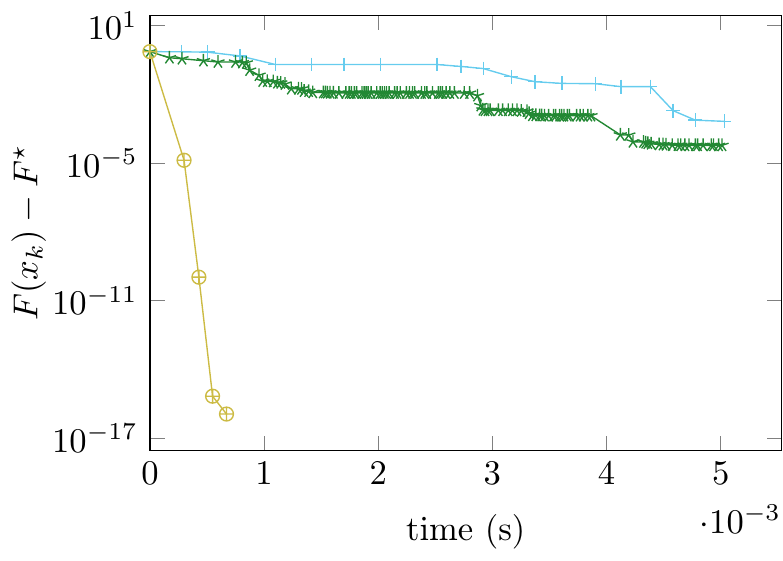}{\vspace*{-3ex}MaxQuad}
  \end{subfigure}\hfill
  \begin{subfigure}[t]{0.49\textwidth}
    \centering
    \includetikzexpenumcaption{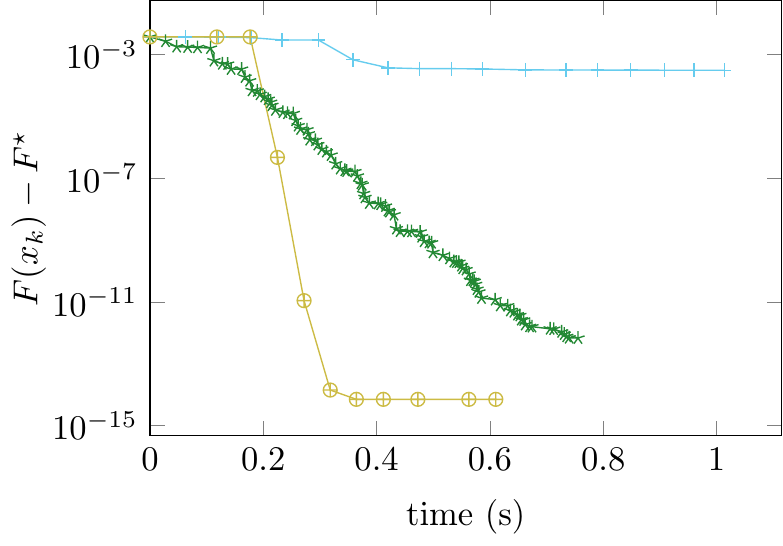}{\vspace*{-3ex}Eigmax affine}
  \end{subfigure} \\ \vspace{0.2cm}
  \begin{subfigure}[t]{\textwidth}
    \centering
    \includetikzexpenum[0.6]{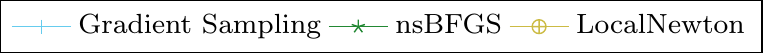}
  \end{subfigure} %
  \caption{Suboptimality vs time (s)\label{fig:suboptvstime}\vspace*{-3.5ex}}
\end{figure}

\begin{figure}[h]
  \centering
  \begin{subfigure}[t]{0.49\textwidth}
    \centering
    \includetikzexpenumcaption{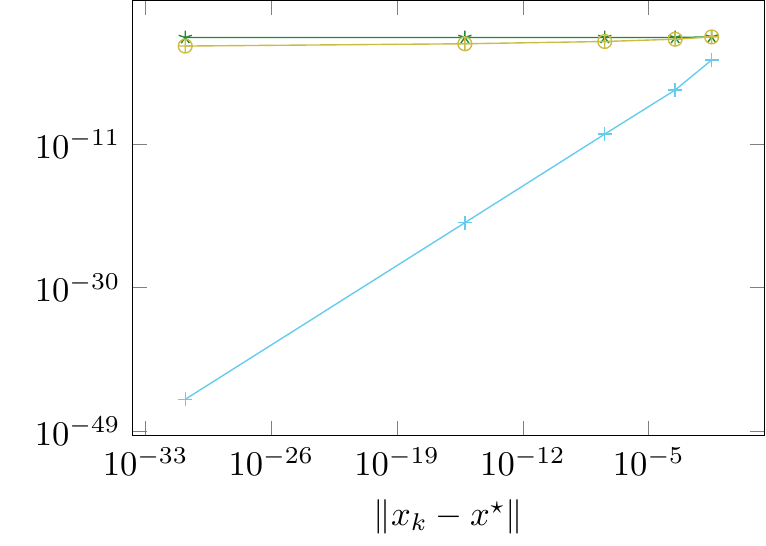}{\vspace*{-2ex}MaxQuad}
  \end{subfigure}\hfill
  \begin{subfigure}[t]{0.49\textwidth}
    \centering
    \includetikzexpenumcaption{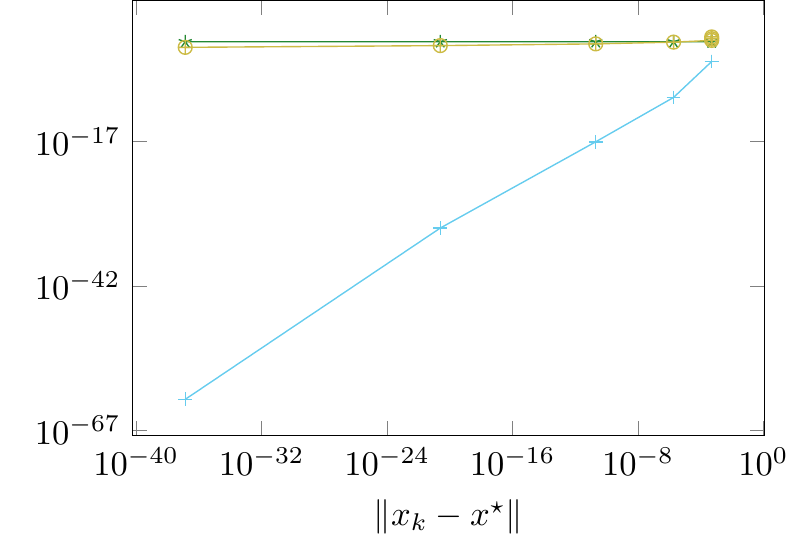}{\vspace*{-2ex}Eigmax affine}
  \end{subfigure} \\ \vspace{0.2cm}
  \begin{subfigure}[t]{\textwidth}
    \centering
    \includetikzexpenum[0.35]{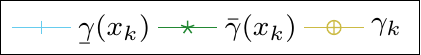}
  \end{subfigure} %
  \caption[Stepsize vs iteration]{Stepsize $\curr[\step]$ vs iteration \label{fig:stepitBigFloat}\vspace*{-2ex}}
\end{figure}

\section{Conclusions}
\label{sec:conclusions}

This paper studies the local structure of functions that write as a composition of a nonsmooth function with a smooth mapping.
When the proximity operator of the nonsmooth function is explicitly available, we show that the structure of the minimizer can be detected.
We further use this information to propose a local Newton method to minimize the objective harnessing the detected structure. This method is guaranteed to identify the structure of the minimizer and to converge quadratically. We illustrate this behavior on two standard nonsmooth problems. %

\appendix

\section{The maximum and maximum eigenvalue satisfy the normal ascent and curve properties}%
\label{sec:addasmproof}
We show here that the maximum and the maximum eigenvalue meet the normal ascent~\Cref{prop:normalascent} and curve properties~\Cref{prop:curve}.
We begin with a lemma that simplifies verification of~\Cref{prop:curve}.

\begin{lemma}\label{lemma:simplecurveproof}
  Consider a function $\funns$, partly smooth at a point $\crit[\vxinter]$ relative to a manifold $\Minter$, and a smooth application $\cm:\N_{\crit[\vxinter]}\times[0, \Tcurve] \to \Minter$ defined for a neighborhood $\N_{\crit[\vxinter]}$ of $\crit[\vxinter]$ and $\Tcurve>0$ such that $\cm(\vxinter,0) = \projM[\Minter](\vxinter)$, $\frac{\dd}{\dd t}e(\vxinter, t)\vert_{t=0} = -\grad \funns(\projM[\Minter](\vxinter))$.

  \noindent If $\D \left(t \mapsto \proj_{\normal{\cm(\vxinter, t)}{\Minter}}(\projM(\vxinter) - \vxinter)\right) = 0$ for all $\vxinter\in\N_{\crit[\vxinter]}$, then $\funns$ satisfies~\Cref{prop:curve} at point $\crit[\vxinter]$.
\end{lemma}
\begin{proof}
  We denote $\theta(\vxinter, t) = \proj_{\normal{\cm(\vxinter, t)}{\Minter}}(\cm(\vxinter, t) - \vxinter)$. First,
  \begin{align}
      \frac{\dd}{\dd t}\theta(\vxinter, t)\vert_{t=0} = &\D \left(t \mapsto \proj_{\normal{\cm(\vxinter, t)}{\Minter}}(\projM[\Minter](\vxinter) - \vxinter)\right) \\
      &~~ + \proj_{\normal{\projM(\vxinter)}{\Minter}} \left(\D (t \mapsto (\cm(\vxinter, t) - \vxinter))(0)
    \right),
  \end{align}
  where the first term is null by assumption and the second is also null since it is the normal projection of the tangent vector $\grad \funns(\projM[\Minter](\vxinter))$. Thus, $\frac{\dd}{\dd t}\theta(\vxinter, t)\vert_{t=0} = 0$.   Using this fact and smoothness of $\theta$, Taylor's theorem with Lagrange remainder yields, for all $\vxinter\in\N_{\crit[\vxinter]}$, the existence of $\bar{t}\in [0, \Tcurve]$ such that, for all $t\in [0, \Tcurve]$,
  \begin{align}
    \theta(\vxinter, t) = \theta(\vxinter, 0) + \frac{t^{2}}{2} \frac{\dd^{2}}{\dd t^{2}}\theta(\vxinter, \bar{t}).
  \end{align}
  Therefore, for all $\vxinter \in \N_{\crit[\vxinter]}$ and $t \in [0, \Tcurve]$,
  \begin{align}
    \|\theta(\vxinter, t)\| \le \|\theta(\vxinter, 0)\| + \frac{t^{2}}{2} \sup_{\bar{t} \in [0, \Tcurve]} \frac{\dd^{2}}{\dd t^{2}}\theta(\vxinter, \bar{t}) \le \|\theta(\vxinter, 0)\| + t^{2} \constcurve,
  \end{align}
  where $\constcurve = \sup_{\vxinter \in \N_{\crit[\vxinter]}} \sup_{\bar{t} \in [0, \Tcurve]} \frac{\dd^{2}}{\dd t^{2}}\theta(\vxinter, \bar{t})$.
\end{proof}

We can now proceed with the proof of \cref{lemma:props}, divided into two parts corresponding to the two cases of the result.
The case $\funns = \max$ comes easily, due to the polyhedral nature of the function.

\begin{lemma}\label{lemma:maxofsmoothprops}
  Consider $\funns = \max$, a point $\crit[\vxinter] \in \interSpace$ and the corresponding structure manifold $\MI$ (of~\Cref{ex:maxofsmoothstruct}).
  Then \Cref{prop:normalascent,prop:curve} hold at $\crit[\vxinter]$.
\end{lemma}

\begin{proof}
  \emph{Normal ascent}
  Take $\vxinter \in \MI$ for some active indices $I\subset \{1, \ldots, m\}$.
  A normal direction $d\in\normal{\vxinter}{\MI}$ is such that $d_{i} = 0$ for $i\notin I$ and $\sum_{i\in I}d_{i} = 0$.
  Thus $\max (\vxinter + td) = \vxinter_{i}+td_{i}$ with $i = \argmax_{i} d_{i}$, and $\D\max(\vxinter)[d] = \lim_{t\searrow 0}(\max(\vxinter+t d)-\max(\vxinter))/t = d_{i}>0$ for all $d\neq 0$.

  \emph{Curve assumption}
  Since the structure manifold of $\max$ are affine subspaces, the normal spaces are equal at all points of the manifold.
  Therefore the derivative of the projection at a parametrized point is null and~\Cref{lemma:simplecurveproof} provides the result.
\end{proof}

The case $\funns = \lammax$ is not difficult \emph{per se}, but requires a precise description of the geometry of the maximum eigenvalue function and its structure manifolds; we refer to~\cite{shapiroEigenvalueOptimization1995a,oustryLagrangianMaximumEigenvalue1999} for the derivation of these tools.

\begin{lemma}\label{lemma:eigmaxprops}
  Consider $\funns = \lammax$, a point $\crit[\vxinter] \in \Sym$ and the corresponding structure manifold $\Mr$ (of~\Cref{ex:eigmaxstruct}).
  Then \Cref{prop:normalascent,prop:curve} hold at $\crit[\vxinter]$.
\end{lemma}

\begin{proof}
  \emph{Normal ascent}
  Take $\vxinter\in \Mr$, let $U\in\bbR^{m\times r}$ denote a basis of the first eigenspace of matrix $\vxinter$ and $d\in \normal{\vxinter}{\Mr}$.
  The normal space at $\vxinter\in\Mr$ writes (\cite[Th. 4.3, Cor. 4.8]{oustryLagrangianMaximumEigenvalue1999})
  \begin{align}
    \normal{\vxinter}{\Mr} = \{ U(\vxinter)ZU(\vxinter)^{\top}, Z \in \Sym[r], \trace(Z) = 0\}.
  \end{align}
  Therefore, $d = UZU^{\top}$ for some $Z\in\Sym[r]$ such that $\trace(Z) = 0$.
  Let $s = U (I/r + \alpha Z) U^{\top}$ where $\alpha >0$ is small enough so that $s$ is positive definite.
  Since $s$ has also unit trace, it is a subgradient of $\lammax$ at $\vxinter$~\cite[Th. 4.1]{oustryLagrangianMaximumEigenvalue1999}.
  Thus $\lammax'(\vxinter; d) = \sup_{v\in\partial \lammax(\vxinter)}\langle v, d \rangle \ge \langle s, d \rangle = \langle I/r + \alpha Z, Z \rangle = \alpha \|Z\|^{2}$\GBreplace{.
  Hence}{, which yields} $\lammax'(\vxinter; d) >0$ for any $d\in \normal{\vxinter}{\Mr}\setminus\{0\}$.

  \emph{Curve assumption}
  Let $\crit[\vxinter] \in \Mr$. For any $\vxinter \in \Sym$, we denote by $P(\vxinter)$ the orthogonal projection on the eigenspace corresponding to the $r$ largest eigenvalues of $\vxinter$ (counting multiplicities).
  This operator is smooth.
  We can define a mapping $U:\Sym\to\bbR^{m\times r}$ such that: $U(\vxinter)^{\top}U(\vxinter) = I_{r}$, $P(\vxinter) = U(\vxinter)U(\vxinter)^{\top}$, $U$ is smooth near our reference point $\crit[\vxinter]$ and its derivative at $\crit[\vxinter]$ satisfies $\D U(\crit[\vxinter])^{\top} U(\crit[\vxinter]) = 0$.
  The mapping $U$ defines a smooth orthonormal basis of the eigenspace corresponding to the $r$ largest eigenvalues~\cite[p. 557]{shapiroEigenvalueOptimization1995a}.
  Finally, for a point $\vxinter'\in\Mr$, the projection of $d\in\Sym$ on $\normal{\vxinter'}{\Mr}$ writes
  \begin{align}
    \proj_{\normal{\vxinter'}{\Mr}}\left( d \right) = U(\vxinter') \left\{U(\vxinter')^{\top}d U(\vxinter') - \frac{1}{r}\trace ( U(\vxinter')^{\top}d U(\vxinter') )I_{r}\right\} U(\vxinter')^{\top}.
  \end{align}

  Now, fix $\vxinter$ near $\crit[\vxinter]$, consider the eigenbasis $U$ with reference point $\cm(\vxinter, 0)=\projM[\Mr](\vxinter)$.
  Following~\Cref{lemma:simplecurveproof}, let $\nu: t \mapsto \proj_{\normal{\cm(\vxinter, t)}{\Mr}}(d)$ with $d = \projM[\Mr](\vxinter) - \vxinter$.
  We can now give an explicit expression of $\nu(t)$ and show that $\frac{\dd}{\dd t}\nu(0)$ is null.
  Denoting $U(t) = U(\cm(\vxinter, t))$, we have
  \begin{align}\label{eq:curveproof}
    \nu(t) = U(t) \underbrace{\left\{U(t)^{\top}d U(t) - \frac{1}{r}\trace ( U(t)^{\top}d U(t) )I_{r}\right\}}_{\defeq \chi(t)} U(t)^{\top}.
  \end{align}

  First, as $d$ is a normal vector to $\Mr$ at point $\projM[\Mr](\vxinter)$, there exists $Z\in\Sym[r]$ such that $d = U(0) Z U(0)^{\top}$.
  Using that $\D U(0)^{\top} U(0) = 0$ yields
  \begin{align}
    \D U(0)^{\top} d U(0) = \D U(0)^{\top} U(0) Z U(0)^{\top} U(0) = 0.
  \end{align}
  Then, one readily checks that $U(0) \D \chi(0) U(0) = 0$.

  We turn to the term $\D U(0) \chi(0) U(0)^{\top}$.
  A quick computation from the eigen decomposition of $\vxinter$ shows that $d$ writes $U(0) Z U(0)^{\top}$, where $Z$ is actually diagonal.
  Therefore, $\chi(0) = Z - (1/r)\trace(Z) I_{r}$ is a diagonal matrix, so that
  \begin{align}
    \D U(0) \chi(0) U(0)^{\top} = \sum_{i=1}^r \chi(0)_{ii} \D U_{i}(0) U_{i}(0)^{\top}.
  \end{align}
  Following~\cite{shapiroEigenvalueOptimization1995a}, the differential of $t \mapsto U(\cm(\vxinter, t))$ at $t=0$ writes
  \begin{align}
    \D U_{i}(0) = \sum_{k=r+1}^m \frac{1}{\lambda_{1} - \lambda_{k}} U_{k}(0) U_{k}(0)^{\top} \eta U_{i}(0),
  \end{align}
  with $\eta = \grad \lammax(\projM[\Mr](\vxinter))$.
  Using that $\lammax(y) = (1/r) \sum_{i=1}^r U_{i}(y)^{\top} y U_{i}(y)$, we compute the Riemannian gradient (see~\cite[Sec.~7.7]{boumal2022intromanifolds}):
  \begin{equation}
    \grad \lammax(y) = \frac{1}{r} \sum_{i=1}^r U_{i}(y)^{\top}U_{i}(y).
  \end{equation}
  By orthogonality of the smooth basis of eigenvectors, the terms $U_{k}(0)^{\top}U_{i}(0)$ vanish for all $i\in \{1, \ldots, r\}$ and $k \in \{r+1, \ldots, m\}$.
  We get that $\D U(0) \chi(0) U(0)^{\top} = 0$, and thus that $\D \nu(0) = 0$.
  Thus,   \Cref{lemma:simplecurveproof} applies and yields the result.
\end{proof}

\section{Numerical experiments in high precision}
\label{apx:highprec} We report in~\Cref{fig:suboptvstimeBigFloat} the evolution of suboptimality versus computing time, for the same problems and algorithms as in \cref{sec:numexps}, but with a high precision floating type. Indeed, the flexibility of the Julia language allows to use the same implementation with the high precision \texttt{BigFloat} type, which precision is $1.73\cdot 10^{-72}$, or the usual \texttt{Float64} type, which precision is $2.22\cdot 10^{-16}$.

\begin{figure}[h]
  \centering
  \begin{subfigure}[t]{0.49\textwidth}
    \centering
    \includetikzexpenumcaption{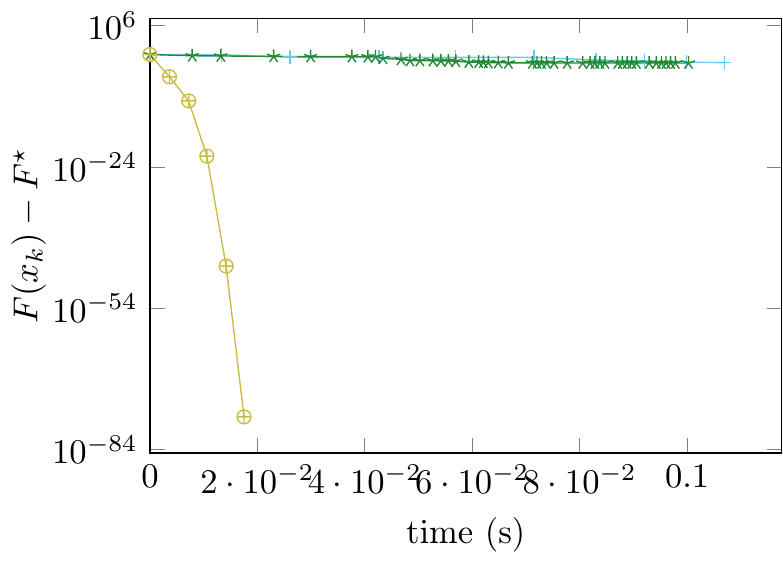}{\vspace*{-2ex}MaxQuad}
  \end{subfigure}\hfill
  \begin{subfigure}[t]{0.49\textwidth}
    \centering
    \includetikzexpenumcaption{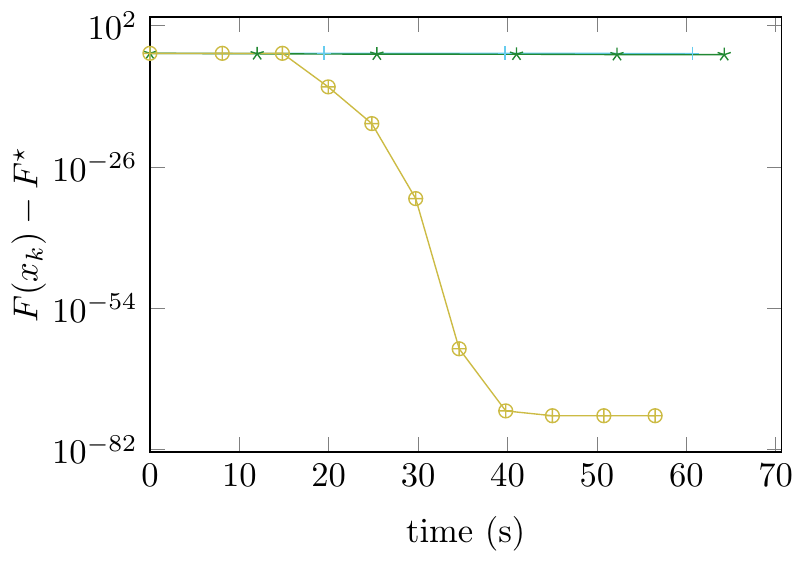}{\vspace*{-2ex}Eigmax affine}
  \end{subfigure} \\ \vspace{0.2cm}
  \begin{subfigure}[t]{\textwidth}
    \centering
    \includetikzexpenum[0.6]{legend_subopt_time}
  \end{subfigure} %
  \caption{Suboptimality vs time (s)\label{fig:suboptvstimeBigFloat}\vspace*{-2ex}}
\end{figure}

\section*{Acknowledgments}
This work is funded by the ANR JCJC project STROLL (ANR-19-CE23-0008) and MIAI@Grenoble Alpes (ANR-19-P3IA-0003). \FIedit{We thank the three anonymous referees and the associate editor for their improvement suggestions that lead to a better readability and exposition of the paper.}

\bibliographystyle{siamplain}
\bibliography{references}

\end{document}